\documentclass[11pt]{article}
\usepackage{amsmath,amsthm,amssymb}
\usepackage[T1]{fontenc}
\usepackage[utf8]{inputenc}
\usepackage[dvipdf]{graphicx}
\usepackage{color}
\usepackage{epstopdf}
\usepackage{dsfont}
\usepackage{enumerate}
\usepackage{enumitem}
\numberwithin{equation}{section}
\usepackage[normalem]{ulem}

\definecolor{db}{RGB}{0, 0, 130}
\usepackage[colorlinks=true,citecolor=red,linkcolor=db,urlcolor=blue,pdfstartview=FitH]{hyperref}

\usepackage[numbers]{natbib}

\definecolor{rp}{rgb}{0.25, 0, 0.75}
\definecolor{dg}{rgb}{0, 0.6, 0}

\newtheorem{theorem}{Theorem}[section]

\newtheorem{definition}{Definition}[section]

\newtheorem{example}[definition]{Example}
\newtheorem{assumption}[theorem]{Assumption}
\newtheorem{lemma}[definition]{Lemma}

\newtheorem{proposition}[definition]{Proposition}
\newtheorem{remark}[definition]{Remark}
\def\R{\mathbb{R}}

\def\E{\mathbb{E}}

\def\Bh{\widehat{B}}
\def\Q{\mathbb{Q}}
\def\x{\times}
\def\Om{\Omega}
\def\om{\omega}
\def\Fc{\mathcal{F}}
\def\F{\mathbb{F}}
\def\P{\mathbb{P}}
\def\Xb{\overline{X}}
\def\eps{\varepsilon}

\def\Yh{\widehat{Y}}
\def\Tc{\mathcal{T}}
\def\Lh{\widehat{L}}
\def\Lc{\mathcal{L}}

\def\Yc{\mathcal{Y}}

\def\tauh{\hat \tau}

\def\taut{\tilde \tau}

\def\esssup{{\rm ess}\!\sup}
\def\essinf{{\rm ess}\!\inf}

\def \proof{{\noindent \bf Proof }}
\def \endproof{\hbox{ }\hfill$\Box$}

\title{An exit contract optimization problem}
\author{
Xihao He
\footnote{Department of Mathematics, The Chinese University of Hong Kong. xihaohe@math.cuhk.edu.hk.}
\and
Xiaolu Tan
\footnote{Department of Mathematics, The Chinese University of Hong Kong. xiaolu.tan@cuhk.edu.hk. Research supported by CUHK startup grant, and by Hong Kong RGC General Research Fund (Projects 14302921 and 14302622).}
\and
Jun Zou
\footnote{Department of Mathematics, The Chinese University of Hong Kong. zou@math.cuhk.edu.hk. Research substantially supported by Hong Kong RGC General Research Fund 
(Projects 14306921 and 14306719). }
}
\date{\today}

\begin{document}

\maketitle

\begin{abstract}

	We study an exit contract design problem, where one provides a universal exit contract to multiple heterogeneous agents, with which each agent chooses an optimal (exit) stopping time.
	The problem consists in optimizing the universal exit contract w.r.t. some criterion depending on the contract as well as the agents' exit times.
	Under a technical monotonicity condition, and by using Bank-El Karoui's representation of stochastic processes, we are able to transform the initial contract optimization problem into an optimal control problem.
	The latter is also equivalent to an optimal multiple stopping problem and the existence of the optimal contract is proved.
	We next show that the problem in the continuous-time setting can be approximated by a sequence of discrete-time ones, which would induce a natural numerical approximation method.
	We finally discuss the optimization problem over the class of all Markovian and/or continuous exit contracts.
\end{abstract}


\vspace{0.5em}

\noindent\textbf{MSC2010 subject classification:} 60G40, 60G07, 49M25

\section{Introduction}

	The contract theory is an important research subject in economics and applied mathematics.
	A basic problem in the theory is about the design of a contract between the principal(s) and the agent(s),
	where the principal provides to the agent a contract (reward function) depending on the action
and/or the corresponding output of the agent, while the agent takes an optimal action according to the given contract/reward. The problem (of the principal) is then to design a best contract by taking into account the (optimal) action of the agent as well as the corresponding output and reward value paid to the agent.
	In this paper, we introduce an exit contract optimization problem with $n \ge 1$ heterogeneous agents in a stochastic context.
	Let $(\Om, \Fc, \F, \P)$ be a filtered probability space, an exit contract is mathematically an adapted reward process $Y = (Y_t)_{t \ge 0}$,
	with which each agent $i =1, \cdots, n$ solves an optimal stopping problem
	$$
		\sup_{\tau \in \Tc} ~ \E \Big[ \int_0^{\tau} f_i(t) dt + Y_{\tau} \Big],
	$$
	in order to choose an optimal exit time $\tauh_i$, where $\Tc$ represents the collection of all stopping times.
	In above, the agents share the universal contract $Y$, but each agent has individual utility/reward process $f_i$ for $i=1, \cdots, n$.
	Our problem consists in choosing an optimal contract $Y$ by considering the optimization probelm:
	$$
		\sup_{Y} ~G(Y, \tauh_1, \cdots, \tauh_n),
	$$
	where $G$ is a reward function depending on the contract $Y$ as well as the corresponding optimal stopping times $(\tauh_i)_{i=1, \cdots, n}$ of the agents.
	
	\vspace{0.5em}
	
	The above problem can be considered as a variation of the so-called principal-agent problem,
	where a principal provides a contract w.r.t. which the agent solves a stochastic optimal control problem,
	and then the principal optimizes the contract w.r.t. a reward functional depending on the contract as well as the agents' optimal control.
	In the continuous-time setting, the principal-agent problem has been studied by characterizing the optimal behavior of the agents with the first order necessary condition, see e.g. Cvitanić and Zhang \cite{CZ} and the references therein.
	More recently, a dynamic programming approach has been introduced by Sannikov \cite{Sannikov}, and then developped by Cvitanić, Possama\"i and Touzi \cite{CPT}
	with the 2nd order BSDE (backward stochastic differential equation) technique of Soner, Touzi and Zhang \cite{STZ} (see also Possama\"i, Tan and Zhou \cite{PTZ}).

	The optimal stopping mechanism can also be included in this problem, see, e.g.,
	Cvitanić, Wan and Zhang \cite{CWZ2}, Sannikov \cite{Sannikov}, Capponi and Frei \cite{CapponiFrei2015}, Hajjej, Hillairet and Mnif \cite{HHM}, and Lin, Ren, Touzi and Yang \cite{LinRenTouziYang}, in the setting with one agent.
	While the initial principal-agent problem focused on the case with one principal and one agent,
	Elie and Possama\"i \cite{ElieP}, Elie, Mastrolia and Possama\"i \cite{EMP},
	Ren, Tan, Touzi and Yang \cite{RenTTY}
	studied the case with one principal and many agents, and
	Mastrolia and Ren \cite{MastroliaRen2018}, Hu, Ren and Yang \cite{HRY} investigated
	the case with many principals and one agent.
	Kang \cite{Kang2013} {considered} the problem with several principals and several agents.
	{We like also to mention another interesting work by} Hernández Santibáñez, Possamaï and Zhou \cite{HsPossamaiZhou}, where the principal offers a family of contracts and {the agent can choose one.}
	{For the applications of the principal-agent problem in such as energy or financial market,
	we refer to, among many others,
	the work by A\"id, Possama\"i and Touzi \cite{APT}, Elie, Hubert, Mastrolia and Possama\"i \cite{EHMP},
	Euch, Mastrolia Rosenbaum and Touzi \cite{EuchMastroliaRosenbaumTouzi}, Baldacci, Possama\"i and Rosenbaum \cite{BPR}.}
	
	\vspace{0.5em}

	By adopting this language, we will also call our exit contract optimization problem the principal's problem.
	For our exit contract problem, a main difference is that each agent needs to solve an optimal stopping problem in contrast to an optimal control problem in the principal-agent problem.
	More importantly, the principal can observe the agents' optimal choice (exit time) in our problem, while the principal can not observe the agents' optimal control in the standard principal-agent problem.
	In other words, there is no moral hazard in our exit contract design problem.
	
	\vspace{0.5em}
	
	Further, in our exit contract design problem,
	although the principal has full information on the reward process $f_i$ of each agent $i=1, \cdots, n$, she/he needs to propose a universal contract $Y$ to all agents, rather than an individual contract to each agent.
	This should be considered as a constraint on the principal's problem,
	{which occurs in various applications.} Let us mention some of them below.
	\begin{itemize}
		\item The exit contract process $Y$ can represent the price process of the good/service (e.g. the electricity) sold publicly.
		In this case, the price could depend on the time, but it needs to be the same to all customers.
		
		\item In some countries, when a company plans to lay off a number of employees in a progressive way, the manager can not directly fire the chosen employees
		because of the labour union or law constraint.
		One needs to provide a time-dependent compensation plan which is identical for everyone, and then let the employees take voluntary leave.
		
		\item In some retirement {systems,} people are allowed to choose {their retirement ages in a range}.
		In particular, everyone has an individual pension basis, and the real retirement pension depends on the basis as well as a {parameter}.
		The retirement plan, i.e. how this {parameter} depends on the age, needs to be universal to all people or all groups of people.
	\end{itemize}

	Let us also mention the recent work by Nutz and Zhang \cite{NZ}, which studies an optimal design problem of the exit scheme for a (large) population of agents playing a mean-field game of  stopping under the exit scheme.
	In their context, each agent runs an identical and independent Brownian motion, and the agents are in a symmetric situation with interaction.
	For our exit contract problem, the main difference is that all agents share the same randomness in the probability space $(\Om, \Fc, \P)$ and there is no interaction between the agents.

	\vspace{0.5em}

	In the context with only one agent, the exit contract problem could become trivial as the principal can easily manipulate the exit time of the agent,
	so that it turns to be a so-called first best problem for the optimal stopping part,
	see e.g. Cvitanić, Wan and Zhang \cite{CWZ2}.
	However, in the current context with a universal exit contract for multiple heterogeneous agents, the agents may have different (optimal) exit times as their utility/reward process $(f_i)_{i = 1, \cdots, n}$ are different.
	It becomes no more trivial to manipulate directly all exit times of multiple agents as in the one agent case.
	To the best of our knowledge, this formulation (in both discrete-time and continuous-time settings) has not been studied in the literature.
	To provide a first approach, we follow the spirit of Sannikov \cite{Sannikov} to focus on the value processes of the agents, which allows decoupling the principal's and the agents' problems.
	In particular, we will apply the remarkable representation theorem of the stochastic processes due to Bank and El Karoui \cite{BankKaroui2004} (see also the recent development by Bank and Besslick \cite{Bank2018}).
	By this representation theorem,
	any optional process $Y = (Y_t)_{t \in [0,T]}$ satisfying some integrability and regularity in expectation conditions can be represented as an integral of functional of another optional process $L = (L_t)_{t \in [0,T]}$.
	More importantly, the hitting times of the process $L$ at different level provide the (minimal) solutions of a family of optimal stopping problems relying on $Y$.
	From this point of view, one can use $L$ to represent the contract $Y$ and at the same time the optimal exit times of the agents.
	It follows that, at least formally, one can decouple the initial problems and reformulate the exit contract design problem as an optimization problem over a class of processes $L$.
	Moreover, the exit design problem can be then reformulated to a multiple optimal stopping problem, where the optimal solution corresponds to the optimal stopping times of the agents with the corresponding optimal exit contract.
	From this point of view, our approach confirms that, under our technical conditions,
	our exit contract problem with multiple agents can be transformed as a first best problem (with a well chosen reward function) as in the one agent case (see more discussions in Section \ref{subsec:discussion}).

	\vspace{0.5em}

	In the continuous-time setting, some technical upper-semicontinuous condition is required on the admissible contracts $Y$ to ensure the existence of the optimal exit time of the agents.
	In the discrete-time setting, such technical condition is not required in the classical optimal stopping theory.
	We hence provide an analogue of the main results on the exit contract design problem, by developing the same techniques in the discrete-time framework.
	We can also show the convergence of the discrete-time problems to the continuous-time problem as time step tends to $0$.
	In particular, this could induce natural numerical approximation methods for the initial continuous-time exit contract problem.
	Finally, by Bank-El Karoui's representation theorem in the discrete-time setting,
	we can identify the properties of the process $L$ when $Y$ is Markovian and/or continuous w.r.t. some underlying process $X$,
	which allows us to study the contract design problem when one is restricted to choose from the class of all Markovian and/or continuous contracts.

	\vspace{0.5em}

	The rest of the paper is organized as follows.
	In Section \ref{sec:PA_continuous}, we formulate our exit contract design problem in a continuous-time framework, and provide an approach to decouple the principal's and agents' problems based on Bank-El Karoui's representation theorem.
	It is shown that the principal's problem (i.e. exit contract design problem) is equivalent to an optimal control problem or a multiple optimal stopping problem, and the existence of the optimal contract is obtained.
	Some examples and interpretations are provided in the end.
	In Section \ref{sec:PA_discrete}, we develop the analogue techniques for the discrete-time problem, and show its convergence to the continuous-time one as time step goes to $0$.
	Finally, in Section \ref{sec:Markovian}, we stay in the discrete-time setting and study the problem with Markovian and/or continuous contracts.

\section{Exit contract optimization problem in continuous-time setting}
\label{sec:PA_continuous}

	Let $(\Om, \Fc, \P)$ be a complete probability space, equipped with a filtration $\F = (\Fc_t)_{t \in [0,T]}$, for some finite $T> 0$, satisfying the usual conditions,
	i.e. the map $t \longmapsto \Fc_t$ is right-continuous, and $\Fc_0$ contains all null sets in $\Fc$.
	Let $\Tc$ denote the collection of all $\F$--stopping time taking values in $[0,T]$.
	Given $\tau \in \Tc$, we also denote
	$$
		\Tc_{\tau} ~:=~ \big\{ \theta \in \Tc ~: \theta \geq \tau \big\}.
	$$

\subsection{Mathematical formulation of the exit contract problem}
	We will formulate an exit contract design problem with $n$ agents.
	An exit contract is a $\F$-optional process $Y = (Y_t)_{t \in [0,T]}$,
	and each agent chooses to quit the contract at time $t \in [0,T]$ to receive a reward value $Y_t$.
	In the typical case where $\F$ is the (augmented) filtration generated by some c\`adl\`ag process $X = (X_t)_{t \in [0,T]}$, a contract $Y$ is a functional of the underlying process $X$.
	The contract design problem consists in choosing an optimal contract $Y$ w.r.t. a criterion depending on $Y$ as well as the agents' exit times.

\subsubsection{Agents' problem}

	Let us consider $n$ agents indexed by $i=1, \cdots, n$.
	Given a fixed contract $Y = (Y_t)_{t \in [0,T]}$, which is a $\F$-optional process of class (D) (i.e. the family $(Y_{\tau})_{\tau \in \Tc}$ is uniformly integrable),
	the agent $i$ aims at solving the following optimal stopping problem
	\begin{equation} \label{eq:AgentPb}
		V^A_i
		~:=~
		\sup_{\tau \in \Tc}~ \E \Big[ \int_{[0,\tau]} f_i(t) \mu^A(dt) + Y_{\tau} \Big],
	\end{equation}
	where $f_i : [0,T] \x \Om \longrightarrow \R$ is a given (progressively measurable) reward function, $\mu^A$ is a deterministic atomless finite Borel measure on $[0,T]$.
	Namely, the agent $i$ chooses an exit time $\tau_i$, before which she/he would receive a reward value with rate $f_i(t)$ (w.r.t. $\mu^A$),
	and at which she/he would receive the compensation $Y_{\tau_i}$.
	We assume that all agents stay in a risk-neutral context under $\P$,
	so that their optimization problems turn out to be \eqref{eq:AgentPb}.
	
	\vspace{0.5em}
	
	The optimal stopping problem \eqref{eq:AgentPb} can be solved by the classical Snell envelop approach (see the recalling in Theorem \ref{thm:SnellEnv}).
	Namely, let us denote
	$$
		G^i_t ~:=~ \int_{[0,t]} f_i(s) \mu^A(ds) + Y_t,
	$$
	and by $S^i$ the Snell envelop of $G^i$, so that, for all $\tau \in \Tc$,
	\begin{equation}\label{eq:prop_SnellEnv}
		S^i_{\tau}
		=
		\int_{[0,{\tau}]} f_i(s) \mu^A(ds)
			+ Z^{A, i}_{\tau}, ~\mbox{a.s.},
		~~\mbox{with}~
		Z^{A,i}_{\tau}
		:=
		\esssup_{\sigma \in \Tc_{\tau}}
		\E \bigg[ \int_{(\tau,{\sigma}]} f_i(s) \mu^A(ds)
		+ Y_{\sigma} \Big|\Fc_\tau \bigg].
	\end{equation}
	Then $V^A_i = \E[ S^i_0] = \E[ Z^{A,i}_0]$.

	\begin{definition}[USCE] \label{def:USCE}
		An optional process $Y$ of class $(D)$ is said to be upper-semicontinuous in expectation ($\mathrm{USCE}$) if, for any $\tau \in \Tc$ and any sequence $(\tau_n)_{n \ge 1} \subset \Tc$ satisfying either
		$$
			(\tau_n)_{n\ge 1} ~\mbox{is non-decreasing,}~
			\tau_n < \tau ~\mbox{on}~\{\tau > 0\},~\mbox{for each}~n\ge 1,
			~\lim_{n \to \infty} \tau_n = \tau,
		$$
		or
		$$
			\tau_n \ge \tau,~\mbox{a.s., for each}~ n \ge 1,~\lim_{n \to \infty} \tau_n = \tau,
		$$
		one has
		$$
			\E[Y_{\tau}] ~\geq~ \limsup_{n\to \infty}\E[Y_{\tau_n}].
		$$
	\end{definition}
	By the classical optimal stopping theory (see e.g. Theorem \ref{thm:SnellEnv}),
	when the optional process $Y$ is of class (D) and USCE in the sense of Definition \ref{def:USCE},
	the smallest optimal stopping time is given by
	\begin{equation} \label{eq:def_tauh_i}
		\tauh_i
		~:=~
		\inf \big\{ t \ge 0 ~:  S^i_t = G^i_t \big\}
		~=~
		\essinf \big\{ \tau \in \Tc ~: Z^{A,i}_{\tau} = Y_\tau, ~\mbox{a.s.} \big\}.
	\end{equation}
	We will assume the above technical condition on the admissible contract $Y$,
	so that each agent has a unique smallest optimal stopping time.

\subsubsection{The exit contract optimization problem}

	Let $\xi$ be a fixed $\Fc_T$--measurable random variable such that $\E[|\xi|] < \infty$,
	and define
	\begin{equation*}
		\Yc
		~:=~
		\big\{
			Y~\mbox{is}~\F\mbox{--optional, USCE, in class (D), and satisfies}~ Y_T \ge \xi
		\big\}.
	\end{equation*}
	Then for each $Y \in \Yc$, the optimal stopping problem \eqref{eq:AgentPb} of the agent $i$ has a unique smallest optimal stopping time $\tauh_i$ given by \eqref{eq:def_tauh_i}.
	To make the behaviour of the agents tractable, we fix $\Yc$ as the set of all admissible contracts,
	and also assume that the agents will choose to exit the contract at the smallest optimal stopping time among all optimal ones.

	\vspace{0.5em}
	
	Next, before the exit time $\tauh_i$ of agent $i$, the principal will receive a (continuous) reward value with rate $g_i(t)$
	(w.r.t a deterministic atomless finite Borel measure $\mu^P$ on $[0,T]$) due to the work of agent $i$, where $g_i: [0,T] \x \Om \longrightarrow \R$ is a progressively measurable process.
	Further, at the exit time $\tauh_i$, the principal pays the agent $i$ the compensation $Y_{\tauh_i}$.
	We also assume that the principal is risk-neutral, so that the exit contract optimization problem turns out to be
	\begin{equation} \label{eq:PrincipalPb}
		V^P
		~:=~
		\sup_{Y \in \Yc} \E \bigg[ \sum_{i=1}^n \Big( \int_{[0,\tauh_i]} g_i(t) \mu^P(dt) - Y_{\tauh_i} \Big) \bigg].
	\end{equation}

	\begin{remark} \label{rem:Y_xi}
	$\mathrm{(i)}$ The constraint $Y_T \ge \xi$ can be considered as the participation constraint in our exit contract problem.
		In particular, it ensures that the reward value $V^A_i$ of agent $i$ satisfies
		$$
			V^A_i  ~\ge~ C_i ~:=~ \E \Big[ \int_0^T f_i(t) \mu^A(dt) + \xi \Big].
		$$
		Notice that the optimal stopping times $(\tauh_i)_{i=1,\cdots, n}$ stay unchanged if the principal {replaces} the contract $(Y_t)_{t \in [0,T]}$ by $(Y_t- C)_{t \in [0,T]}$ for
		an arbitrarily big constant $C > 0$, which would make {the} reward value $V^P$ of the principal in \eqref{eq:PrincipalPb} to be $\infty$.
		The constraint $Y_T \ge \xi$ will prevent the principal to choose the contract in this way.
		
	\vspace{0.5em}
	
	\noindent $\mathrm{(ii)}$ Furthermore, the optimal exit times $(\tauh_i)_{i=1,\cdots, n}$ of the agents will not change if the principal replaces $Y$ by $Y - M$ for some martingale $M$.
	Therefore, the principal will always choose a contract $Y$ such that $Y_T = \xi$.
	Otherwise, he/she can replace $Y$ by $Y - M$ with $M_t := \E[ Y_T -\xi | \Fc_t]$ for $t \in [0,T]$ to have a better reward value.	
	For this reason, one can assume, {additionally and w.l.o.g.}, that $Y_T = \xi$ for all admissible contracts in $\Yc$.
	\end{remark}

	\begin{remark}
		One can also consider different atomless measures $\{\mu^P_i\}_{i=1}^n$ for different agents in principal's problem.
		Nevertheless, by considering a measure $\mu^P$ dominating all measures $\mu^P_i$ and then modifying the rate functions $g_i$,
		one can still reduce principal's problem to  that in \eqref{eq:PrincipalPb}.
	\end{remark}

	In the following of the paper, we make the following technical assumptions on $(f_i, g_i)_{i=1, \cdots, n}$.
	\begin{assumption}\label{assum:f}
		It holds that
		\begin{equation}\label{eq:monotone_cond}
			f_1(t, \om) < \cdots < f_n(t, \om),  ~~\mbox{for all}~(t, \om) \in [0,T] \x \Om,
		\end{equation}
		and
		$$
			\E \Big[ \int_0^T \big| f_i(t) \big| \mu^A(dt)
				+
				\int_0^T \big|g_i(t)\big| \mu^P(dt)
			\Big] < \infty,
			~~
			i = 1, \cdots, n.
		$$
	\end{assumption}

	\begin{remark}
		The monotonicity condition \eqref{eq:monotone_cond} is a technical condition due to our approach,
		which would be restrictive in practice.
		In particular, it implies that the optimal stopping times $\{\hat \tau_i\}_{i=1}^n$ are ordered, see also Remark \ref{rem:ordered_times} below.
	\end{remark}

\subsection{Solving the exit contract problem}

	We will make use of Bank-El Karoui's representation theorem (recalled in Theorem \ref{thm:BErepresentationTheorem}) to solve the principal's problem \eqref{eq:PrincipalPb}.
	As preparation, let us first interpolate the functionals $(f_i)_{i = 1, \cdots, n}$ as a functional defined on $[0,T] \x \Om \x \R$.
	By abus of notation, we denote it by $f: [0, T] \x \Om \x \R \longrightarrow \R$: $~\mbox{for all}~(t, \om) \in [0,T] \x \Om$
	\begin{align} \label{eq:f_interpolation}
		f(t, \om, \ell) ~:=~
		\begin{cases}
			f_1(t, \om) + (\ell - 1), & \ell \in (-\infty,1], \\
			(i+1 - \ell) f_i(t, \om) + (\ell - i) f_{i+1}(t, \om), ~& \ell \in [i, i+1], ~i = 1,\cdots, n-1, \\
			f_n(t, \om) + (\ell - n), &\ell \in [n,+\infty).
		\end{cases}
	\end{align}

	Then under Assumption \ref{assum:f}, it is clear that, for each fixed $\ell \in \R$,
	$$
		\E \bigg[ \int_{[0,T]} \big| f \big(t, \ell \big) \big| \mu^A(dt) \bigg] < \infty.
	$$
	Further, for each $(t, \om) \in [0,T] \x \Om$, the map $\ell \longmapsto f(t, \om, \ell)$ is continuous and strictly increasing from $-\infty$ to $+\infty$,
	and finally, $f(t, \om, i) = f_i(t, \om)$ for each $i = 1, \cdots, n$.

	\vspace{0,5em}

	Applying  Bank-El Karoui's representation theorem (Theorem \ref{thm:BErepresentationTheorem}),
	it follows that, for every contract $Y \in \Yc$, there exists an optional process $L^Y$ such that
	\begin{equation}\label{eq:BErepresentationPA}
		Y_\tau
		~=~
		\E \bigg[\xi + \int_{(\tau,T]}f\Big(t, \sup_{s\in [\tau,t)}L^Y_s \Big)\mu^A(dt) \Big| \Fc_\tau\bigg],
		~~\mbox{for all}~
		\tau \in \Tc.
	\end{equation}
	Moreover, the (smallest) optimal stopping time of each agent $i=1, \cdots, n$ can be obtained from the process $L^Y$ by
	\begin{equation} \label{eq:def_tauh_i}
		\hat \tau_i
		~=~
		\inf \big\{ t\geq 0 ~: L^Y_t \geq i \big\}
		=
		\inf \big\{ t \ge 0 ~: Y_t = Z^{A,i}_t \big\}.
	\end{equation}
	Notice that in \eqref{eq:BErepresentationPA}, there may exist {multiple} solutions $L^Y$ for a given $Y$, but their running maximum $\widehat L^Y_t := \sup_{0 \le s \le t} L^Y_s$ will be the same, so that the stopping times $\hat \tau_i$ are uniquely defined.
	Then, at least formally at this stage, one can expect to reformulate the exit contract optimization problem \eqref{eq:PrincipalPb} as an optimization problem over a class of optional processes $L$.
	Let us denote by $\Lc$ the collection of all $\F$-optional processes $L: [0,T] \x \Om \longrightarrow [0, n]$,
	and the subsets
	\begin{equation}\label{Definition:Lc}
		\Lc^+ :=
		\big\{ L \in \Lc ~: L ~\mbox{is non-dereasing} \big\},
		~~
		\Lc^+_0
		:=
		\big\{
			L \in \Lc^+ ~:
			L_t \in \{0, 1, \cdots, n \}, t \in [0,T],~\mbox{a.s.}
		\big\}.
	\end{equation}

	Before we provide the solution of the principal's problem based on the  representation in \eqref{eq:BErepresentationPA},
	we show that given any $L \in \Lc^+$, one can construct an admissible contract $Y^L \in \Yc$.

	\begin{lemma}\label{L-definedY}
		Let Assumption \ref{assum:f} hold true.
		Then for every $L \in \Lc^+$, there exists an optional process $Y^L$ such that
		\begin{equation} \label{eq:L2Y}
			Y^L_{\tau}
			~=~
			\E \Big[\xi + \int_{(\tau,T]}f\Big(t, \sup_{s\in [\tau,t)}L_s \Big)\mu^A(dt) \Big| \Fc_\tau\Big],
			~\mbox{a.s., for all}~
			\tau \in \Tc.
		\end{equation}
		Moreover, $Y^L \in \Yc$, and it has almost surely right-continuous paths.
	\end{lemma}

	\begin{proof}
	$\mathrm{(i)}$ Notice that a process $L \in \Lc^+$ is nondecreasing, so that
	\begin{equation} \label{eq:L2Y_nosup}
		\E\Big[\xi + \int_{(\tau,T]} f\Big(t, \sup_{s\in [\tau, t)} L_s \Big)\mu^A(dt) \Big| \Fc_\tau \Big]
		~=~
		\E\Big[\xi + \int_{(\tau,T]} f\big(t, L_{t-} \big)\mu^A(dt) \Big|\Fc_\tau\Big].
	\end{equation}
	Therefore, there exists a c\`adl\`ag adapted  (and hence optional) process $Y^L$ satisfying \eqref{eq:L2Y}.
	
	\vspace{0.5em}
	
	\noindent $\mathrm{(ii)}$ We next prove that $Y^L \in \Yc$. From \eqref{eq:L2Y}, we first obtain that $Y^L_T = \xi$.
	Moreover, $Y^L$ is of Class (D) since
	\begin{eqnarray*}
		|Y^L_\tau|
		&=&
		\Big|\E\Big[\xi + \int_{(\tau,T]}
		f\Big(t, \sup_{s\in [\tau,t)} L_s \Big)
		\mu^A(dt)\Big|\Fc_\tau\Big] \Big| \\
		&\leq&
		\E\Big[|\xi| + \int_{[0,T]}
		\big(|f_1(t)| + 1 + |f_n(t)|\big)
		\mu^A(dt) \Big|\Fc_\tau\Big].
	\end{eqnarray*}
	It is enough to check that $Y^L$ is USCE in order to conclude that $Y^L \in \Yc$.
	It is in fact continuous in expectation.
	Indeed, let $\{\tau_k\}_{k=1}^\infty \subset \Tc$ be such that $\tau_k \longrightarrow \tau$ as $k \to \infty$ for some $\tau \in \Tc$,
	then we have by \eqref{eq:L2Y} and \eqref{eq:L2Y_nosup}
	\begin{align*}
		\E \big[ Y^L_{\tau_k} - Y^L_\tau \big]
		~=~&
		\E\Big[\int_{(\tau_k,T]}
        f\Big(t, \sup_{s\in [\tau_k, t)}L_s \Big)\mu^A(dt) - \int_{(\tau,T]}
        f\Big(t, \sup_{s\in [\tau,t)}L_s \Big)\mu^A(dt)\Big]
		\\
		~=~&
		\E\big[\int_{(\tau_k, \tau]} f\big(t, L_{t-}  \big)\mu^A(dt) \Big]
		~\longrightarrow~
		0,
	\end{align*}
	where the last limit follows from the fact that
	$$
		\Big| f\Big(t, L_{t-}  \Big) \Big|
		~\leq~
		|f_1(t)| + |f_n(t)| + 1
	$$
	and the integrability condition on $(f_i)_{i = 1, \cdots, n}$ in Assumption \ref{assum:f}.
	\end{proof}

	\begin{theorem}\label{ReductionByBErepresentation}
		Let Assumption \ref{assum:f} hold true.
		Then the contract design problem \eqref{eq:PrincipalPb} is equivalent to
		\begin{align}\label{eq:PrincipalPb-L}
			V^P
			& ~=
			\sup_{L \in \Lc^+}
			\E\Bigg[ \sum_{i=1}^{n} \bigg(
			\int_{[0,T]}\Big( g_i(t)\mathds{1}_{\{L_{t-} < i\}}\mu^P(dt) - f\big(t,L_{t-}\big)\mathds{1}_{\{L_{t-} \geq i\}}\mu^A(dt) \Big)  - \xi \bigg) \Bigg]   \\
			&~=
			\sup_{L \in \Lc^+_0}
			\E\Bigg[\sum_{i=1}^{n} \bigg(
			\int_{[0,T]} \Big(g_i(t)\mathds{1}_{\{L_{t-} < i\}}\mu^P(dt) - f\big(t,L_{t-}\big)\mathds{1}_{\{L_{t-} \geq i\}}\mu^A(dt) \Big) - \xi \bigg)\Bigg]. \nonumber
		\end{align}
		Further, their optimal solutions are also related in the following sense:

		\vspace{0.5em}
		
		\noindent $\mathrm{(i)}$
		Let $\Lh^* \in \Lc^+$ be an optimal solution to the optimization problem at the r.h.s. of \eqref{eq:PrincipalPb-L}, then the contract $Y^{\Lh^*}$, defined below, is an optimal solution to the contract design problem \eqref{eq:PrincipalPb}:
		\begin{equation}\label{YdefinedL}
			Y^{\Lh^*}_t
			~:=~
			\E\Big[\xi + \int_{(t,T]}f\big(s, \Lh^*_{s-}\big)\mu^A(ds) \Big|\Fc_t\Big],
			~~ t \in [0,T].
		\end{equation}
		Moreover, the smallest optimal stopping times $(\tauh^*_i)_{i=1, \cdots, n}$ of the agents w.r.t. the contract $Y^{\Lh^*}$ is given by
		$$
			\tauh^*_i ~:=~ \inf \big\{ t \ge 0~:\Lh^*_t \ge i\big\}.
		$$
		
		\noindent $\mathrm{(ii)}$
		Conversely, let $Y^*$ be an optimal contract  to Problem \eqref{eq:PrincipalPb},
		and $L^{Y^*}$ be a corresponding optional process in the representation \eqref{eq:BErepresentationPA},
		let
		$$
			\Lh^{Y^*}_t
			~:=~
			\Big(0 \vee \sup_{s \in [0,t)} L^{Y^*}_s \Big)\wedge n.
		$$
		Then $\Lh^{Y^*} \in \Lc^+$ and it is an optimal solution to the optimization problem at the r.h.s. of \eqref{eq:PrincipalPb-L}.
		Moreover, the contract $\Yh$ defined by  $\Lh^{Y^*}$ through \eqref{YdefinedL} is also an optimal contract to \eqref{eq:PrincipalPb},
		and it induces the same smallest optimal stopping time $\tauh_i$ (defined in \eqref{eq:def_tauh_i}) for each agent $i = 1, \cdots, n$ as those induced by contract $Y^*$.
	\end{theorem}

	\begin{remark}
		There are different ways to interpolate $(f_i)_{i = 1, \cdots, n}$ to obtain a functional $f: [0,T] \x \Om \x \R \longrightarrow \R$ as in \eqref{eq:f_interpolation}.
		Nevertheless, when we consider the processes  in $\Lc^+_0$, which take values only in $\{0, 1, \cdots, n\}$,
		it does not change the problem at the r.h.s. of \eqref{eq:PrincipalPb-L}.
	\end{remark}

	\proof \textbf{of Theorem \ref{ReductionByBErepresentation}.}
	$\mathrm{(i)}$
	Let us first prove that
	\begin{equation}\label{Inequality:reduction-leq}
		V^P
		~\leq~
		\sup_{L \in \Lc^+}
		\E\Bigg[\sum_{i=1}^{n} \bigg(
			\int_{[0,T]} \Big( g_i(t)\mathds{1}_{\{L_{t-} < i\}}\mu^P(dt) - f\big(t, L_{t-} \big)\mathds{1}_{\{L_{t-} \geq i\}}\mu^A(dt) \Big) - \xi \bigg)
		\Bigg].
	\end{equation}
	Given any $Y\in \Yc$, we denote by $L^Y$ an optional process which provides the representation \eqref{eq:BErepresentationPA},
	and then define
	$$
		\Lh^{Y}_t
		~:=~
		0 \vee \sup_{s\in [0,t)}L^Y_s \wedge n,
		~~~ t\in [0,T].
	$$
	One observes that, for each $i = 1, \cdots, n$,
	$$
		\{\tauh_i < t\}
		~\subset~
		\big\{ i \leq \Lh^{Y}_t \big \}
		~\subset~
		\{\tauh_i\leq t\},
		~~\mbox{and}~~
		\sup_{s\in [\tauh_i,t)}L^Y_s = \sup_{s\in [0,t)}L^Y_s
		~\mbox{on}~
		\{\tauh_i < t \}.
	$$
	It follows that
	\begin{align}\label{eq:tau2L}
		J^P_i (Y)
		~&:=~
		\E \Big[  \int_{[0,\tauh_i]} g_i(t) \mu^P(dt) - Y_{\tauh_i} \Big ]  \nonumber \\
		&=~
		\E \Big[
			\int_{[0,\tauh_i]} g_i(t)\mu^P(dt)
- \int_{(\tauh_i,T]}f\Big(t, \sup_{s\in [\tauh_i,t)}L^Y_s \Big) \mu^A(dt) - \xi
		\Big] \nonumber \\
        &=~
		\E \Big[
			\int_{[0,\tauh_i]} g_i(t)\mu^P(dt) - \int_{(\tauh_i,T]}f\Big(t, \sup_{s\in [0,t)}L^Y_s \Big) \mu^A(dt) - \xi
		\Big] \nonumber \\
		&\leq~
		\E \Big[
			\int_{[0, T]} \Big( g_i(t) \mathds{1}_{\{t \leq \tauh_i\}}\mu^P(dt) -  f\big(t, \Lh^{Y}_t \big)\mathds{1}_{\{t > \tauh_i\}} \mu^A(dt)\Big) - \xi
		\Big] \nonumber \\
		& = ~
		\E \Big[
			\int_{[0,T]} \Big( g_i(t) \mathds{1}_{\{\Lh^{Y}_t < i\}}\mu^P(dt) - f\big(t, \Lh^{Y}_t \big)\mathds{1}_{\{\Lh^{Y}_t \geq i\}}  \mu^A(dt)\Big) - \xi
		\Big].
	\end{align}
	Notice that $\Lh^Y \in \Lc^+$ has almost surely left continuous paths, taking the sum over $i=1, \cdots, n$,
	it follows that \eqref{Inequality:reduction-leq} holds.
	
	\vspace{0.5em}
	
	\noindent $\mathrm{(ii)}$
	We next prove the reverse inequality:
	\begin{equation}\label{Inequality:reduction-geq}
		V^P
		~\geq~
		\sup_{L \in \Lc^+}
		\E\Bigg[\sum_{i=1}^{n} \bigg(
		\int_{[0,T]}\Big(g_i(t)\mathds{1}_{\{L_{t-} < i\}}\mu^P(dt) - f\big(t, L_{t-} \big)\mathds{1}_{\{L_{t-} \geq i\}}\mu^A(dt) \Big) - \xi \bigg) \Bigg] .
	\end{equation}

	For each $L \in \Lc^+$, let $Y^L$ be the optional process given by Lemma \ref{L-definedY}, so that $Y^{L} \in \Yc$ and
	$$
		Y^{L}_\tau
		~=~
		\E\Big[\xi + \int_{(\tau,T]} f\Big(t, \sup_{s\in [\tau,t)} L_s \Big)\mu^A(dt) \Big| \Fc_\tau\Big],
		~\mbox{a.s.},
		~~\mbox{for all}~\tau \in \Tc,
	$$
	where $\sup_{s \in [\tau,t)}L_s = L_{t-}$, since $L \in \Lc^+$ admits nondereasing paths.
	Let
	$$
		\tauh_i := \inf\{t\geq 0 ~:L_t \geq i\},
	$$
	then by Bank-El Karoui's representation theorem (Theorem \ref{thm:BErepresentationTheorem}),
	$\tauh_i$ is the smallest optimal stopping time  of the $i$-th agent under the contract $Y^{L}$.
	In particular, one has
	\begin{equation}\label{eq:L2tau}
		\E \bigg[ \int_{[0,T]}f\big(t,L_{t-}\big)\mathds{1}_{\{t > \tauh_i\}}\mu^A(dt) \bigg]
		~=~
		\E \bigg[ \int_{[0,T]}f\big(t,L_{t-}\big)\mathds{1}_{\{L_{t-} \geq i\}}\mu^A(dt) \bigg].
	\end{equation}
	It follows that
	\begin{equation*}
		V^P
		~\geq~
		\sum_{i=1}^n J^P_i(Y^L)
		=
		\E\Bigg[\sum_{i=1}^{n} \bigg(
		\int_{[0,T]}\Big(g_i(t)\mathds{1}_{\{L_{t-} < i\}}\mu^P(dt) - f\big(t,L_{t-}\big)\mathds{1}_{\{L_{t-} \geq i\}}\mu^A(dt) \Big) - \xi \bigg) \Bigg],
	\end{equation*}
	and therefore the inequality in \eqref{Inequality:reduction-geq} holds.

	\vspace{0.5em}

	\noindent $\mathrm{(iii)}$
	Notice that for any $L \in \Lc^+$, we define $L^0_t := [L_t]$, where $[x]$ denotes the biggest integer less or equal to $x$.
	It is easy to verify that $L^0 \in \Lc^+$, $L^0 \leq L$ and $\{L_{t-} \geq i\} = \{L^0_{t-}\geq i\}$ for any $t\in [0,T]$, $i = 1, \cdots, n$.
	Hence we have
	\begin{align*}
		&\E\Bigg[\sum_{i=1}^{n} \bigg(
		\int_{[0,T]}\Big(g_i(t)\mathds{1}_{\{L_{t-} < i\}}\mu^P(dt) - f\big(t, L_{t-}\big)\mathds{1}_{\{L_{t-} \geq i\}}\mu^A(dt) \Big) - \xi \bigg)\Bigg] \\
		\leq~ &\E\Bigg[\sum_{i=1}^{n} \bigg(
		\int_{[0,T]}\Big(g_i(t)\mathds{1}_{\{L^0_{t-} < i\}}\mu^P(dt) - f\big(t,L^0_{t-}\big)\mathds{1}_{\{L^0_{t-} \geq i\}}\mu^A(dt) \Big) - \xi \bigg)\Bigg].
	\end{align*}
	So we complete the proof of the second equality in \eqref{eq:PrincipalPb-L}.

	\vspace{0.5em}

	\noindent $\mathrm{(iv)}$
	Given an optimal solution $\Lh^* \in \Lc^+$ to the optimization problem at the r.h.s. of \eqref{eq:PrincipalPb-L}, together with the equality \eqref{eq:L2tau},
	one has
	\begin{eqnarray*}
		V^P
		&=&
		\E\Bigg[ \sum_{i=1}^{n} \bigg(
			\int_{[0,T]}\Big(g_i(t)\mathds{1}_{\{\Lh^*_{t-} < i\}}\mu^P(dt) - f\big(t, \Lh^*_{t-} \big)\mathds{1}_{\{\Lh^*_{t-} \geq i\}}\mu^A(dt) \Big) - \xi \bigg)
		\Bigg] \\
		&=&
		\E \bigg[ \sum_{i=1}^{n} \bigg( \int_{[0,\tauh^*_i]} g_i(t) \mu^P(dt) - Y^{\Lh^*}_{\tauh^*_i} \bigg) \bigg ],
	\end{eqnarray*}
	where $\tauh^*_i$ is the smallest optimal stopping time of agent $i$ with contract $Y^{\Lh^*}$.
	Hence $Y^{\Lh^*}$ is an optimal solution to the exit contract design problem.
	
	\vspace{0.5em}
	
	Similarly, given an optimal solution $Y^*$, by using \eqref{eq:tau2L} and the representation \eqref{eq:BErepresentationPA},
	one can conclude that the corresponding $\Lh^{Y^*}$ is an optimal solution to the optimization problem at the r.h.s. of \eqref{eq:PrincipalPb-L}, $\Lh^{Y^*} \in \Lc^+$,
	and  the contract $\Yh$ defined by  $\Lh^{Y^*}$ through \eqref{YdefinedL} is also an optimal contract to \eqref{eq:PrincipalPb},
	Moreover, since	
	$\{L^{Y^*}_t \ge i\} = \{\Lh^{Y^*}_t \ge i\}$, it follows that $Y^*$ and $\Yh$ induce the same smallest optimal stopping time $\tauh_i$ for each agent $i = 1, \cdots, n$.
	\endproof

	\begin{remark} \label{rem:ordered_times}
		As observed in the above proof,
		given $L \in \Lc^+$ and the corresponding contract $Y^L$ defined in Lemma \ref{L-definedY},
		the optimal stopping time of agent $i$ is given by $\tauh_i := \inf\{t\geq 0 ~:L_t \geq i\}$.
		In particular, they are ordered in the sense that
		$$
			\hat \tau_1 \le \cdots \le \hat \tau_n.
		$$		
		In fact, given an arbitrary contract $Y \in \Yc$, one can check that the smallest optimal stopping time $\tauh_i$ of the agent $i$ in Problem \eqref{eq:AgentPb} satisfies
		\begin{equation} \label{eq:ordered_stopping}
			\tauh_1 ~\leq~ \tauh_2 ~\leq~ \cdots ~\leq~ \tauh_n,
			~
			\mbox{a.s.}
		\end{equation}
		Indeed, under Assumption \ref{assum:f} and by the definition of $Z^{A,i}$ in \eqref{eq:prop_SnellEnv}, it follows that, for all $t \in [0,T]$,
		$$
			Z^{A,1}_t ~\leq~ Z^{A,2}_t ~\leq~ \cdots ~\leq~ Z^{A,n}_t, ~\mbox{a.s.}
		$$
		Notice that $Z^{A,i}_t \ge Y_t$ for all $i = 1, \cdots, n-1$, then
		$$
			Z^{A,i+1}_t = Y_t,~\mbox{a.s.}
			~~\Longrightarrow~~
			Z^{A,i}_{t} = Y_t,~\mbox{a.s.},
		$$
		and it follows then from \eqref{eq:def_tauh_i} that $\tauh_i \le \tauh_{i+1}$, a.s. and hence \eqref{eq:ordered_stopping} holds.
	\end{remark}

	As noticed in Remark  \ref{rem:ordered_times}, the event $\{L_{t-} \ge i\}$ at the r.h.s. of \eqref{eq:PrincipalPb-L} defines a sequence of ordered stopping times.
	One can then reformulate the the exit contract design problem as a multiple stopping problem.
	More importantly, this reformulation allows us to obtain the existence of the optimal contract.

	\begin{proposition} \label{prop:P_SnellEnv}
		Let Assumption \ref{assum:f} hold true.
		
		\vspace{0.5em}
		
		\noindent $\mathrm{(i)}$
		The exit contract design problem \eqref{eq:PrincipalPb} is equivalent to the following optimal multiple stopping problem,
		with $\tau_{n+1} \equiv T$,
		\begin{equation}\label{eq:PrincipalPb_equiv}
			V^P
			~=
			\sup_{\substack{\{\tau_i\}_{i=1}^n \subset \Tc \\ \tau_1 \leq \cdots \leq \tau_n}}
			\E \bigg[\sum_{i=1}^n \Big( \int_{[0,\tau_i]} g_i(t) \mu^P(dt) - \sum_{j = i}^{n}\int_{(\tau_j,\tau_{j+1}]}f_j(t) \mu^A(dt) - \xi \Big) \bigg].
		\end{equation}

		\noindent $\mathrm{(ii)}$
		There exist stopping times $\tauh_1 \le \cdots \le \tauh_n $ which solve the optimization problem at the r.h.s of \eqref{eq:PrincipalPb_equiv}.
		Moreover, the process $\Lh^*_t := \sum_{i=1}^n \mathds{1}_{\{ \tauh_i < t \}}$ is an optimal solution to the optimization problem at the r.h.s. of \eqref{eq:PrincipalPb-L}.
		Consequently, the corresponding contract $Y^{\Lh^*}$ defined as in \eqref{YdefinedL} is an optimal contract to Problem \eqref{eq:PrincipalPb}.
	\end{proposition}
	\begin{proof}
	\noindent $\mathrm{(i)}$ Given any sequence $\{\tau_i\}_{i = 1}^n$ of stopping times satisfying $\tau_1 \leq \cdots \leq \tau_n$,
	one can define an optional process $L \in \Lc^+_0$ by $L_t := \sum_{i = 1}^{n}\mathds{1}_{\{\tau_i < t\}}$.
	On the other hand, given $L \in \Lc^+_0$, one can define a sequence of ordered stopping times by $\tau_i := \inf\{t\geq 0:L_{t-} \geq i\}$, $i=1, \cdots, n$.
	One can then obtain \eqref{eq:PrincipalPb_equiv} from \eqref{eq:PrincipalPb-L}, together with the equality
	\begin{eqnarray*}
		&&
		\int_{[0,T]} g_i(t)\mathds{1}_{\{L_{t-} < i\}}\mu^P(dt) - f\big(t,L_{t-}\big)\mathds{1}_{\{L_{t-} \geq i\}}\mu^A(dt) \\
		&=&
		\int_{[0,\tau_i]} g_i(t) \mu^P(dt) - \sum_{j = i}^{n}\int_{(\tau_j,\tau_{j+1}]}f_j(t) \mu^A(dt).
	\end{eqnarray*}
	
	\noindent $\mathrm{(ii)}$ By arranging the terms at the r.h.s. of \eqref{eq:PrincipalPb_equiv}, it follows that
	\begin{align*}
		~ &  ~
		\E \bigg[\sum_{i=1}^n \Big(
		\int_{[0,\tau_i]}
		g_i(t) \mu^P(dt)
		- \sum_{j = i}^{n}\int_{(\tau_j,\tau_{j+1}]}
		f_j(t) \mu^A(dt)
		- \xi \Big) \bigg]
		\\
		~ = & ~
		\E \bigg[\sum_{i=1}^n \Big(
		\int_{[0,\tau_i]}
		\Big(g_i(t) \mu^P(dt) + \big(if_i - (i - 1)f_{i - 1}\big)(t)
             \mu^A(dt)\Big)
			-
			\int_{[0,T]}
			f_n(t)
			\mu^A(dt)
			- \xi \Big) \bigg].
	\end{align*}
	The stopping time at the r.h.s. of \eqref{eq:PrincipalPb_equiv} is then equivalent to the optimal multiple stopping problem:
	\begin{equation}\label{equiv:MultiOptimalStopping}
			\sup_{\substack{\{\tau_i\}_{i=1}^n \subset \Tc \\ \tau_1 \leq \cdots \leq \tau_n}}
			\E \bigg[\sum_{i=1}^n
                \int_{[0,\tau_i]}
                    \Big(g_i(t) \mu^P(dt) + \big(if_i - (i - 1)f_{i - 1}\big)
                        (t) \mu^A(dt)\Big)
		 \bigg].
	\end{equation}
	At the same time, it is easy to check that the continuity (in expectation) of the mapping
	$$
		t \longmapsto \E \bigg[
		\int_{[0,\tau_i]}
                    \Big(g_i(t) \mu^P(dt) + \big(if_i - (i - 1)f_{i - 1}\big)
                        (t) \mu^A(dt)\Big)
		\bigg].
	$$
	Then it is enough to apply Theorem \ref{thm:OptimalMultiStoppingRecall} to prove the existence of the optimal stopping times $\{\tauh_i\}_{i = 1}^n \subset \Tc$.
	\end{proof}

	\begin{remark}
		The multiple stopping problem has been studied in the classical literature such as in Carmona and Touzi \cite{CarmonaTouzi},
		Kobylanski, Quenez and Rouy-Mironescu \cite{KobylanskiQuenezRouy-Mironescu2011}.
		A main difference of the multiple stopping problem in \eqref{eq:PrincipalPb_equiv} is that the stopping times are required to be ordered.
	\end{remark}

\subsection{Further discussions and examples}
\label{subsec:discussion}

	Our exit contract design problem shares some main features with the classical principal-agent problem (as studied in \cite{Sannikov}) since both problems optimize over a class of contracts.
	However, there would be some structural differences between the two problems:

	\begin{itemize}
		\item For a classical principal-agent problem, a basic structure is the following:
	the agent makes an action $a$, which induces an output $X^a$, and the contract $\xi$ is a function of the output variable $X^a$.
	In this setting, there exist two situations: the principal observes both agent's action $a$ and the output $X^a$, or the principal observes only $X^a$.
	According to the two different situations, the principal's problem would be the so-called first best problem, or the second best problem.

		\item For our exit contract design problem, the agent makes an action $\tau$, and the contract $\xi$ is a function of $(\tau, X)$ with some underlying observable process $X$.
	There would be only one situation: the principal observes the agent's action $\tau$ so that both can agree with the payoff $\xi (\tau, X)$ paid by the principal.
	
	\end{itemize}
	
	For a standard principal agent problem, where the principal can provide an individual contract to the agent,
	it is well-known that the principal's problem can be reduced to the so-called first best problem, i.e. the principal controls directly the action of the agent in order to optimize an appropriate reward function.
	This is also the case for the exit contract design problem in the setting with one agent, see e.g. \cite{CWZ2}.
	However, it becomes much less obvious in our setting where the principal needs to provide a universal contract to multiple heterogeneous agents.
	Our result in Proposition \ref{prop:P_SnellEnv} shows that, under the monotone condition in Assumption \ref{assum:f},
	one can give an order to different agents and then find an appropriate reward function (as that at the r.h.s. of \eqref{eq:PrincipalPb_equiv})
	so that the initial problem reduces to a first best type optimization problem over the sequences of ordered stopping times.
	We have found such a reward function thanks to the approach in Theorem \ref{ReductionByBErepresentation} based on the Bank-El Karoui's representation theorem.
	
	\vspace{0.5em}
	
	Nevertheless, this seems not to be the feature of our problem without the monotone condition in Assumption \ref{assum:f}.
	Indeed, without the monotone condition, there may be different ways to index the agents, and there is no reason that the optimal stopping times of the agents are ordered.
	Moreover, we show in the following example that, without the monotone condition,
	the problem is not equivalent to the first best problem if one uses the r.h.s. of \eqref{eq:PrincipalPb_equiv} as the reward function.

	\begin{example} \label{exam:not_first_best}
{\rm
		Let us consider a deterministic setting with $2$ agents, where  $T= 3$, $\xi = 0$,
		and $\mu^A(dt) = \mu^P(dt) = \delta_{0}(dt) + \delta_{1}(dt) + \delta_{2}(dt)$, so that the problem can be considered as a discrete-time one on the grid $\{0, 1, 2, 3 \}$.
		Let
		$$
			f_1(0) = f_2(0) = 0,
			~~f_1(1) = 1, ~~f_2(1) = 2,
			~~f_1(2) = 2, ~~f_2(2) = 1,
			~~f_1(3) =f_2(3) = 0,
		$$
		and
		$$
			g_1(0) = g_2(0) = 0,
			~~g_1(1) = 1, ~~g_2(1) = -4,
			~~g_1(2) = -\frac{5}{2},~~g_2(2) = \frac{1}{2},
			~~
			g_1(3) = g_2(3) = 0.
		$$
		In this deterministic setting, the stopping times become deterministic times taking values in $\{0, 1, 2, 3\}$,
		and as the natural extension of the r.h.s. of \eqref{eq:PrincipalPb_equiv}, one can guess that the corresponding first best problem would be
		$$
			V_1 = \sup_{\tau_1 \le \tau_2}
			J_1(\tau_1, \tau_2),
			~\mbox{with}~
			J_1(\tau_1, \tau_2) := \Big(
				\sum_{k=0}^{\tau_1} g_1(k) + \sum_{k=0}^{\tau_2} g_2(k)
				- \sum_{k=\tau_1+1}^{\tau_2} f_1(k)
				- 2\sum_{k=\tau_2+1}^3 f_2(k)
			\Big),
		$$
		or
		$$
			V_2 = \sup_{\tau_2 \le \tau_1} J_2(\tau_2, \tau_1),
			~\mbox{with}~
			J_2(\tau_2, \tau_1) := \Big(
				\sum_{k=0}^{\tau_1} g_1(k) + \sum_{k=0}^{\tau_2} g_2(k)
				- \sum_{k=\tau_2+1}^{\tau_1} f_2(k)
				- 2\sum_{k=\tau_1+1}^3 f_1(k)
			\Big).
		$$
		By a direct computation, we have
		$$
			V_1 \ge J_1(1,2) = - \frac92,
			~~\mbox{and}~
			V_2 \ge J_2(0,2) = - \frac92.
		$$
		In fact, by considering all (finitely) possible values of $(\tau_1, \tau_2)$, one can check that $V_1 = V_2 = -\frac{9}{2}$.
		
		\vspace{0.5em}
		
		Next, let us consider the exit contract optimization problem $V^P$ in \eqref{eq:PrincipalPb}.
		In our deterministic setting, it is enough to consider all possible values of $(Y_0, Y_1, Y_2) \in \R^3$ as $Y_3 = \xi = 0$ (see Remark \ref{rem:Y_xi}).
		Let us define
		\begin{align*}
		&
		D^1_0
		~ := ~
		\{(y_0,y_1, y_2) \in \R^3 : y_0 \geq \max\{1 + y_1, 3 + y_2, 3\}\}, \\
		&
		D^1_1
		~ := ~
		\{(y_0,y_1, y_2) \in \R^3  :  y_1 \geq \max\{2 + y_2, 2\}, ~ y_1 > y_0 - 1\}, \\
		&
		D^1_2
		~ := ~
		\{(y_0,y_1, y_2) \in \R^3 :  y_2 \geq 0, ~ 3 + y_2 > \max\{y_0, 1 + y_1 \}\},
		\\
		&
		D^2_0
		~ := ~
		\{(y_0,y_1, y_2) \in \R^3 :  y_0 \geq \max\{2 + y_1, 3 + y_2, 3\}\},
		\\
		&
		D^2_1
		~ := ~
		\{(y_0,y_1, y_2) \in \R^3 :  y_1 \geq \max\{1 + y_2, 1\}, ~ y_1 > y_0 - 2\},
		\\
		&
		D^2_2
		~ := ~
		\{(y_0,y_1, y_2) \in \R^3 : y_2 \geq 0, ~ 3 + y_2 > \max\{y_0, 2 + y_1\}\}.
        \end{align*}		
	One can easily check that, whenever $(Y_0, Y_1, Y_2) \in D^1_0$ (resp. $D^1_1$, $D^1_2$), one has $\tauh_1 = 0$ (resp. $1$, $2$),
	and whenever $(Y_0, Y_1, Y_2) \in D^2_0$ (resp. $D^2_1$, $D^2_2$), one has $\tauh_2 = 0$ (resp. $1$, $2$).
	For $i, j = 0,1,2$, we define
        \begin{align*}
            V^P(i, j)
    		~:=~
    		\sup_{Y \in D^1_i \cap D^2_j}
            \Big( \sum_{k = 1}^i g_1(k)
            +
            \sum_{k = 1}^j g_2(k)
            - Y_i - Y_j
            \Big),
        \end{align*}
        so that
        $$
        	V^P = \max_{0 \le i, j \le 2} V^P(i,j).
        $$
        By a direct computation, one has
        \begin{align*}
            V^P(0,0) = -6,
            ~
            V^P(0,1) = -8,
            ~
            V^P(1,1) = -7,
            ~
            V^P(2,1) = -\frac{13}{2},
            ~
            V^P(2,2) = -5,
        \end{align*}
        and the value of $V^P(i,j)$ not listed above is $-\infty$.
        Therefore, $V^P = V^P(2,2) = -5$, and one optimal contract would be $(\widehat Y_0, \widehat Y_1, \widehat Y_2, \widehat  Y_3) = (0, 0, 0, 0)$.

        \vspace{0.5em}

	In the above example, we observe that $V^P < V_1 = V_2$, and $V_1$ or $V_2$ can not be the corresponding first best problem.
	At the same time, it seems not clear to us how to formulate an appropriate first best problem for $V^P$ in this setting.
}
	\end{example}

	We next provide an example with explicit solutions to the principal's and agents' problems (under the monotone condition),
	which could also illustrate the structure of our exit contract design problem.

	\begin{example}
{\rm
		Let $n=2$, $T= 1$, $\xi = 0$, $\mu^A(dt) = \mu^P(dt)$ be the Lebesgue measure, $f_1(t) \equiv 1$, $f_2(t) \equiv 2$,
		and $g_1, g_2: [0,T] \longrightarrow \R$ are both deterministic functions.
		Let us define $h_2: [0,1] \longrightarrow \R$, and for each $\tau_2 \in [0,1]$, define $h_1(\tau_2, \cdot) : [0, \tau_2] \longrightarrow \R$ by
		$$
			h_2(t) ~:= \int_t^1 f_2(s) ds ~=~ 2-2t,
			~~\mbox{for all}~t \in [0,1],
		$$
		and
		$$
			h_1(\tau_2, t) ~:=~ h_2 (\tau_2) + \int_t^{\tau_2} f_1(s) ds ~=~ h_2(\tau_2) + \tau_2 - t,
			~~\mbox{for all}~t \in [0, \tau_2].
		$$
		Then the r.h.s. of \eqref{eq:PrincipalPb_equiv} becomes
		\begin{equation} \label{eq:exam_tauh}
			\max_{0 \le \tau_1 \le \tau_2 \le 1}
			~\int_0^{\tau_1} g_1(s) ds + \int_0^{\tau_2} g_2(s) ds - h_1(\tau_2, \tau_1) - h_2(\tau_2).
		\end{equation}
		In this deterministic setting, given the functions $f_i$ and $g_i$, one can easily compute an optimizer $(\tauh_1, \tauh_2)$ for \eqref{eq:exam_tauh}.
		Moreover, by Proposition \ref{prop:P_SnellEnv}, an optimal exit contract for the principal becomes
		$$
			\widehat Y_t = h_1(\tauh_2,t)  \mathds{1}_{\{t = \tauh_1\}} + h_2(t)  \mathds{1}_{\{t = \tauh_2\}},
		$$
		and an optimizer $\widehat L$ for \eqref{eq:PrincipalPb_equiv} can be given by
		$$
			\widehat L_t ~:=~ 0 \x \mathds{1}_{\{0 \le t < \tauh_1\}} + 1 \x \mathds{1}_{\{\tauh_1 \le t < \tauh_2\}} + 2 \x \mathds{1}_{\{\tauh_2 \le t \le T\}},
			~~\mbox{for all}~ t\in [0,T].
		$$

}	
	\end{example}		
		
		Let us provide some interpretations of the functions $h_1$, $h_2$ as well as  problem \eqref{eq:exam_tauh}.
		First, under the monotone condition that $f_1 < f_2$, it is known that the smallest optimal stopping time of Agent $1$ will be smaller than that of Agent $2$ (see Remark \ref{rem:ordered_times}).
		Then, for each $t \in [0,1]$, $h_2(t)$ represents the cumulative reward value that  Agent $2$ expects to receive from time $t$ if he/she chooses not to stop before $T$.
		Thus to encourage Agent $2$ to stop at time $t < T$, the principal should provide at least a compensation value $Y_t = h_2(t)$.
		Next, {depending on} the exit time $\tauh_2$ of Agent 2 and the contract value $Y_{\tauh_2} = h_2(\tauh_2)$,
		the value $h_1(\tauh_2, t)$ denotes the cumulative reward value that Agent $1$ expects to receive from time $t$ if he/she chooses not to stop before $\tauh_2$.
		Therefore, to make Agent $1$ stop at time $t < \tauh_2$, the principal should provide at least a compensation value $Y_t = h_1(\tauh_2, t)$.
		It follows that the principal's optimal exit contract problem turns to be equivalent to \eqref{eq:exam_tauh}.
		
		\vspace{0.5em}
		
		Finally, given an optimal solution $\tauh_1 \le \tauh_2$ of \eqref{eq:exam_tauh}, one can further find an increasing function $L: [0,1] \longrightarrow \{0, 1, 2\}$ such that the hitting time of $L$ to the level $1$ (resp. $2$) is the time $\tauh_1$ (resp. $\tauh_2$),
		which is in fact a solution to the optimization problem at the r.h.s. of \eqref{eq:PrincipalPb_equiv}.

	\begin{remark}
		In view of the above interpretation of $h_1$ and $h_2$, one can in fact provide a direct proof of Proposition \ref{prop:P_SnellEnv} based on a Snell envelop type argument (in a general stochastic setting).
		We choose to first prove Theorem \ref{ReductionByBErepresentation} by the arguments based on Bank-El Karoui's representation theorem, and then to deduce Proposition \ref{prop:P_SnellEnv}, for the following reasons.
		First, from a numerical point of view, the formulation in \eqref{eq:PrincipalPb-L} is a quite standard optimal control problem,
		which {can be solved numerically by the control techniques}.
		We will also develop this point of view and provide a convergence result in Section \ref{subsec:convergence}.
		More importantly, the formulation in \eqref{eq:PrincipalPb-L} and the corresponding approach
		are more flexible when one restricts the initial exit contract problem to a subclass of contract, such as the Markovian contract, and/or continuous contracts w.r.t. some underlying processes.
		We will develop this further in Section \ref{sec:Markovian}, see in particular Remark \ref{rem:no_stopping_formulation} for more discussions.
	\end{remark}

	Finally, let us conclude the subsection by another example which highlights the role of the contract $Y$ in incentivizing agents.

	\begin{example}
		Let us consider a stochastic setting with $n=2$ agents, terminal time $T < \infty$, $\xi = 0$, and $\mu^A(dt) = \mu^P(dt) = dt$.
		Let $B$ be a standard Brownian motion, the running rewards $f_1$ and $f_2$ of the two agents are given by
		\begin{equation*}
			f_1(t) := 1 - 2 \Bh_t - t,
			~~~
			f_2(t) := 2 - \Bh_t - t,
			~\mbox{with}~
			\Bh_t := \sup_{0 \le s \le t}|B_s|,
			~\mbox{for all}~
			t \in [0,T].
		\end{equation*}
		Notice that $f_1$ and $f_2$ clearly satisfy the monotone condition \eqref{eq:monotone_cond}.

		\vspace{0.5em}

		In the case without incentive contract, each agent $i=1, 2$ solves the following optimal stopping problem:
		\begin{equation*}
			\sup_{\tau \in \Tc}\E\bigg[\int_{0}^{\tau}f_i(t)dt\bigg].
		\end{equation*}
		Notice that both $t \longmapsto f_1(t)$ and $t \longmapsto f_2(t)$ are strictly decreasing,
		then the unique optimal stopping time $\taut_i$ for agent $i =1, 2$ will be the first time that $f_i(t)$ becomes negative,
		that is,
		\begin{equation*}
			\taut_1 = \inf\{t \ge 0: 2|B_t| + t \ge 1\} \wedge T,
			~~~
			\taut_2 = \inf\{t \ge 0: |B_t| + t \ge 2\} \wedge T.
		\end{equation*}

		Next, we analyze the behavior of the agents when the principal provides a contract,
		which is also optimized w.r.t. the principal's utility functions $g_1$ and $g_2$.
		By Proposition \ref{prop:P_SnellEnv}, it follows that:
		\begin{align} \label{eq:ineq_optimal_contract}
			V^P
		~ & = ~
            \sup_{\substack{\tau_1, \tau_2 \in \Tc \\ \tau_1 \leq \tau_2}}
            \E\bigg[
            \bigg(\int_{0}^{\tau_1}g_1(t)dt - \int_{\tau_1}^{\tau_2}f_1(t)dt - \int_{\tau_2}^{T}f_2(t)dt\bigg)
            +
            \bigg(\int_{0}^{\tau_2}g_2(t)dt - \int_{\tau_2}^{T}f_2(t)dt\bigg)\bigg]
		\nonumber \\ ~ & = ~
            \sup_{\substack{\tau_1, \tau_2 \in \Tc \\ \tau_1 \leq \tau_2}}
            \E\bigg[
            \int_{0}^{\tau_1}\bigg(g_1(t) + 1 - 2\Bh_t - t\bigg)dt
            +
            \int_{0}^{\tau_2}\big(g_2(t) + 3 - t\big)dt
            \bigg]
            +
            2\E\bigg[\int_{0}^{T}\Bh_tdt\bigg]
            -
            4T.
		\nonumber \\ ~ & \le ~
            \sup_{\tau_1 \in \Tc}\E\bigg[
            \int_{0}^{\tau_1}\big(g_1(t) + 1 - 2\Bh_t - t\big)dt\bigg]
            +
            \sup_{\tau_2 \in \Tc}\E\bigg[
            \int_{0}^{\tau_2}\big(g_2(t) + 3 - t\big)dt
            \bigg]
		\nonumber \\ & ~~~~~~~~~~~~~~~~~~~~~~~~~~~~~~~~~~~~~~~~~~~~~~~~~~~~~~~~~~~~~~~~~~~~
            +
            2\E\bigg[\int_{0}^{T}\Bh_tdt\bigg]
            -
            4T.
		\end{align}
		
       		Let  $\tauh_1, \tauh_2 \in \Tc$ be two stopping times solve respectively the optimal stopping problems
		\begin{equation}\label{eq:Principal_stopping_problem_separate}
			\sup_{\tau_1 \in \Tc}\E\bigg[
			\int_{0}^{\tau_1}\big(g_1(t) + 1 - 2\Bh_t - t\big)dt\bigg],
			~~~
			\sup_{\tau_2 \in \Tc}\E\bigg[
			\int_{0}^{\tau_2}\big(g_2(t) + 3 - t\big)dt
			\bigg].
		\end{equation}
		Assume in addition that $\tauh_1 \le \tauh_2$, a.s.
		Then the inequality in \eqref{eq:ineq_optimal_contract} becomes an equality,
		and $(\tauh_1, \tauh_2)$ is an optimal solution to $V^P$.
		Further, by \eqref{YdefinedL}, it leads to the optimal contract $Y^*$ given by
		\begin{equation}\label{eq:example_optimal_contract}
			\begin{split}
			Y^*_t
			~ = & ~
			\E\bigg[
			\int_{t}^{T}\Big(f_1(s)\mathds{1}_{[0,\tauh_2)}(s)
                        + f_2(s)\mathds{1}_{[\tauh_2,T]}(s)\Big) ds
                        - (\tauh_1 - t)^+ \Big|\Fc_t
			\bigg]
			\\ ~ = & ~
			\E\bigg[
			\int_{t}^{T}\Big((1 - 2\Bh_s - t)\mathds{1}_{[0,\tauh_2)}(s)
                        + (2 - \Bh_s - t)\mathds{1}_{[\tauh_2,T]}(s)\Big) ds
                        - (\tauh_1 - t)^+ \Big|\Fc_t
			\bigg].
			\end{split}
		\end{equation}
		Moreover, $\tauh_i$ is the optimal stopping time of the agents $i =1, 2$ with the contract $Y^*$.

		\vspace{0.5em}
		
		In the following, let us consider different examples of the principal's utility function $g_1$ and $g_2$,
		which leads to different optimal contract $Y^*$ as well as different optimal stopping times $\tauh_1$ and $\tauh_2$ of the agents.
		
		\begin{enumerate}
			\item Let $g_1(t) := 2\Bh_t + t$, $g_2 (t):= 1 + t$ for all $t \in [0,T]$.
			Then by \eqref{eq:Principal_stopping_problem_separate} and \eqref{eq:example_optimal_contract},
			the optimal contract is
			$$
				Y^*_t := -\E \Big[\int_{t}^{T}(2\Bh_s+ s)ds \Big|\Fc_t \Big],
				~~t \in [0,T],
			$$
			and the two optimal stopping problems in \eqref{eq:Principal_stopping_problem_separate} reduce to
			\begin{equation*}
				\sup_{\tau_1 \in \Tc}\E[\tau_1],
				~~~
				\sup_{\tau_2 \in \Tc}\E[4\tau_2],
			\end{equation*}
			so that the optimal stopping times $\tauh_1$ and $\tauh_2$ for the agents are given by
			\begin{equation*}
				\tauh_1 ~=~ \tauh_2 ~=~ T.
			\end{equation*}
			One observes that, comparing to the optimal stopping times $\taut_1$ and $\taut_2$ in the setting without contract,
			the contract incentivizes the agents to work for a longer period (until the terminal time in fact).

			\item Let $g_1(t) := -2 + 2\Bh_t + t$, $g_2(t) := -4 + t$ for $t \in [0,T]$.
			By similar computation, one obtains the principal's optimal contract:
			$$
				Y^*_t := \E \Big[\int_{t}^{T}(2 - \Bh_s - s)ds \Big|\Fc_t \Big],
				~~
				t \in [0,T],
			$$
			and the optimal stopping problems in \eqref{eq:Principal_stopping_problem_separate} reduce to
			\begin{equation*}
				\sup_{\tau_1 \in \Tc}\E[-\tau_1],
				~
				\sup_{\tau_2 \in \Tc}\E[-\tau_2],
			\end{equation*}
			so that the optimal stopping times $\tauh_i$ of the agent $i =1 ,2$ are given by
			\begin{equation*}
				\tauh_1 = \tauh_2 = 0.
			\end{equation*}
			In this setting, the contract incentivise the agent to work for a shorter period (stop immediately at the beginning in fact).
			
			\item Let $g_1 := 1 - 4\Bh - 2t$, $g_2 := 3 - 2\Bh - t$.
			The two optimal stopping problems in \eqref{eq:Principal_stopping_problem_separate} reduce to
			\begin{equation*}
				\sup_{\tau_1 \in \Tc}\E \Big[\int_{0}^{\tau_1}(2 - 6\Bh_t - 3t)dt \Big],
				~~~
				\sup_{\tau_2 \in \Tc}\E\Big [\int_{0}^{\tau_2}(6 - 2\Bh_t - 2t)dt \Big] .
			\end{equation*}
			and similarly, the optimal stopping times of the agents are given by
			\begin{equation*}
				\tauh_1 = \inf \Big\{t \ge 0: 2|B_t| + t \ge \frac{2}{3} \Big\}\wedge T,
				~~~
				\tauh_2 = \inf \Big\{t \ge 0: |B_t| + t \ge 3 \Big\} \wedge T.
			\end{equation*}
			Namely, comparing to $\taut_1$ and $\taut_2$ in the case without contract, the incentive contract makes the agent $i=1$ to work for a shorter period, and the agent $i=2$ to work for a longer period.

		\end{enumerate}
	\end{example}

\section{A discrete-time version and its convergence}
\label{sec:PA_discrete}

	We now study a discrete-time version of the exit contract design problem, and provide some analogue results to those in the continuous-time setting.
	In particular, the USCE technical condition is no more required in  the discrete-time setting to define the admissible contracts.
	We next prove its convergence to the continuous-time problem as the time step goes to $0$.

\subsection{A discrete-time version of the exit contract problem}

	Let us consider a partition $\pi = (t_j)_{0 \le j \le m}$ of the interval $[0,T]$, i.e. $0 = t_0 < t_1 < \cdots < t_m = T$,
	and study the exit contract design problem formulated on the discrete-time grid $\pi$.
	We will in fact embed the discrete-time problem into the continuous-time setting,
	and reformulate it as a continuous-time problem by considering piecewise constant processes.

	\vspace{0.5em}

	We stay in the same probability space setting as in the continuous-time case,
	i.e. $(\Om, \Fc, \P)$ being a complete probability space, equipped with the filtration $\F = (\Fc_t)_{t \in [0,T]}$ satisfying the usual conditions.
	We next define the filtration $\F^{\pi} = (\Fc^{\pi}_t)_{t \in [0,T]}$ by
	\begin{equation} \label{eq:def_Fpi}
		\Fc^{\pi}_t := \Fc_{t_j},~~\mbox{for all}~t \in [t_j, t_{j+1}),~~j=0, \cdots, m-1,
		~~\mbox{and}~
		\Fc^{\pi}_T = \Fc_T,
	\end{equation}
	and $\Tc^{\pi}$ denote the collection of all $\F^{\pi}$-stopping times taking values in $\pi$.
	Next, let
	\begin{equation} \label{eq:def_Ycpi}
		\Yc^{\pi}
		:=
		\big\{
		Y~\mbox{is}~\F^{\pi}\mbox{--adapted s.t.}~
		Y_t = Y_{t_j},~t \in [t_j, t_{j+1}),~\E[|Y_{t_j}|] < \infty,~j=0, \cdots, m,
		~\mbox{and}~Y_T \ge \xi
		\big\},
	\end{equation}
	and
	$$
		\mu^{A, \pi} (dt) = \sum_{j=0}^m c^A_j \delta_{t_j}(dt),
		~~\mbox{for some constants}~ c^A_j > 0, ~j=0, \cdots, m.
	$$
	Given a discrete-time contract $Y \in \Yc^{\pi}$, the agents' optimal stopping problems are given by

	\begin{equation} \label{eq:VApi}
		V^{A, \pi}_i
		~:=~
		\sup_{\tau \in \Tc^\pi}
		~
        \E \Big[ \int_{[0,\tau]} f_i(t) \mu^{A, \pi}(dt) + Y_{\tau} \Big],
        ~i = 1,\cdots, n,
	\end{equation}
	whose Snell envelop is also a piecewise constant process,
	so that the minimum optimal stopping time for the $i$-th agent is given by
	$$
		\tauh_i
        ~:=~
        \essinf \big\{ \tau \in \Tc^\pi ~:
        Z^{A,\pi, i}_{\tau} = Y_\tau, ~\mbox{a.s.} \big\}
	$$
	and takes only values in $\pi$ (i.e. $\tauh_i \in \Tc^{\pi}$), where, as in \eqref{eq:prop_SnellEnv},
	$$
		Z^{A, \pi, i}_{\tau}
		~:=~
		\esssup_{\sigma \in \Tc^\pi_{\tau}}
		\E \Big[ \int_{(\tau,{\sigma}]} f_i(s) \mu^{A, \pi} (ds)
		+ Y_{\sigma} \Big|\Fc_\tau \Big].
	$$
	Similarly, let
	$$
		\mu^{P, \pi} (dt) = \sum_{j=0}^m c^P_j \delta_{t_j}(dt),
		~~\mbox{for some constants}~ c^P_j > 0, ~j=0, \cdots, m.
	$$
	The contract design problem is given by
	\begin{equation} \label{eq:PbP_pi}
		V^{P,\pi}
		~:=~
		\sup_{Y \in \Yc^{\pi}} \E \bigg[ \sum_{i=1}^n \Big( \int_{[0,\tauh_i]} g_i(t) \mu^{P,\pi}(dt) - Y_{\tauh_i} \Big) \bigg].
	\end{equation}
	Recall that $\Lc$, $\Lc^+$ and $\Lc^+_0$ are defined above and in \eqref{Definition:Lc},
	we further define $\Lc^{\pi}$ as the set of all $\F^{\pi}$-optional processes $L: [0,T] \x \Om \longrightarrow [0,n]$ which is constant on each interval $[t_j, t_{j+1})$,
	and
	\begin{equation} \label{eq:def_Lc_pi}
		\Lc^{\pi,+} :=
		\Lc^{\pi} \cap \Lc^+,
		~~~~
		\Lc^{\pi,+}_0
		:=
		\Lc^{\pi} \cap \Lc^+_0.
	\end{equation}
	In this context, one still has the extension of the Bank-El Karoui's representation theorem (see Theorem \ref{thm:BErepresentationTheorem}) without the USCE condition on $Y$.
	It follows that, for any $Y \in \Yc^{\pi}$, there exists an $\F^{\pi}$-optional process $L$ which is piecewise constant on each interval $[t_j, t_{j+1})$ such that
	\begin{equation}\label{eq:YrepresentedLdiscrete}
		Y_\tau = \E\Big[\int_{(\tau,T]} f\Big(t, \sup_{s \in [\tau,t)}L_s \Big) \mu^{A, \pi}(dt) + \xi \Big| \Fc^{\pi}_\tau\Big],
		~\mbox{a.s., for all}~
		\tau \in \Tc,
	\end{equation}
	where $\sup_{s \in [\tau,t)}L_s$ will be replaced by $L_{t-}$ if $L \in \Lc^{\pi,+}$,
	and the hitting time $\tau_{i} :=  \inf \big\{ t\geq 0 ~: L_t \geq i \big\}$ is the minimum solution to the $i$-th agent's optimal stopping problem \eqref{eq:VApi}.

	\vspace{0.5em}
	
	On the other hand, for any $L \in \Lc^{\pi,+}$, the process
	\begin{equation}\label{eq:LdefinedYdiscrete}
		Y^L_{s}
		~:=~
		\E \Big[\xi + \int_{[0,T]}f\Big(t, L_{t-} \Big)\mu^{A, \pi} (dt) \Big| \Fc_s\Big]
		-
		\int_{[0,s]}f\Big(t, L_{t-} \Big)\mu^{A, \pi} (dt),
		~\mbox{a.s.},
	\end{equation}
	satisfies clearly that $Y^L_s = Y^L_{t_j}$ for all $s \in [t_j, t_{j+1})$, $j=0, \cdots, m-1$, so that $Y^L \in \Yc^{\pi}$.
	
	\vspace{0.5em}
	
	One can follow almost the same arguments in Theorem \ref{ReductionByBErepresentation} and Proposition \ref{prop:P_SnellEnv},
	but use the discrete-time version of the optimal stopping theory and Bank-El Karoui's representation theorem (see Theorems \ref{thm:SnellEnv} and \ref{thm:BErepresentationTheorem}),
	to obtain the analogue solution to the discrete-time exit contract problem.
	Let us just state the results and omit the proof.

	\begin{theorem} \label{thm:main_discrete}
		Assume that $\E[ |f_i(t_j)| ] + \E[ |g_i(t_j)|] < \infty$ for all $i=1, \cdots, n$ and $j=0, \cdots, m$,
		and
		$$
			f_1(t_j, \om) < \cdots < f_n(t_j, \om),  ~~\mbox{for all}~(t_j, \om) \in \pi \x \Om.
		$$
		Then one has the following equivalence for the contract design problem:
		\begin{align}\label{eq:PrincipalPb-L_pi}
			V^{P,\pi}
			& ~=
			\sup_{L \in \Lc^{\pi,+}}
			\E\Bigg[\sum_{i=1}^{n} \bigg(
			\int_{[0,T]}\Big(g_i(t)\mathds{1}_{\{L_{t-} < i\}}\mu^{P, \pi} (dt) - f\big(t,L_{t-}\big)\mathds{1}_{\{L_{t-} \geq i\}}\mu^{A, \pi} (dt) \Big) -  \xi \bigg) \Bigg]  \nonumber \\
			&~=
			\sup_{L \in \Lc^{\pi,+}_0}
			\E\Bigg[\sum_{i=1}^{n} \bigg(
			\int_{[0,T]}\Big(g_i(t)\mathds{1}_{\{L_{t-} < i\}}\mu^{P, \pi} (dt) - f\big(t,L_{t-}\big)\mathds{1}_{\{L_{t-} \geq i\}}\mu^{A, \pi} (dt) \Big) -  \xi \bigg) \Bigg]  \nonumber \\
			&~=
			\sup_{\substack{\{\tau_i\}_{i=1}^n \subset \Tc^{\pi} \\ \tau_1 \leq \cdots \leq \tau_n}}
			\E \bigg[\sum_{i=1}^n \bigg( \int_{[0,\tau_i]} g_i(t) \mu^{P,\pi}(dt) - \sum_{j = i}^{n}\int_{(\tau_j,\tau_{j+1}]}f_j(t) \mu^{A, \pi} (dt) - \xi \bigg) \bigg].
		\end{align}
		Moreover, there exists an optimal contract for Problem \eqref{eq:PbP_pi}.
	\end{theorem}

	\begin{remark}
		As in Theorem \ref{ReductionByBErepresentation} and Proposition \ref{prop:P_SnellEnv},
		one can also construct an optimal solution to \eqref{eq:PbP_pi} from a solution to \eqref{eq:PrincipalPb-L_pi},
		and vice-versa.
		We nevertheless skip this for simplicity.
	\end{remark}

\subsection{Convergence of the discrete-time value function to the continuous-time one}
\label{subsec:convergence}

	As illustrated in Section 5 of Bank and F\"ollmer \cite{BankFollmer}, the discrete-time version of the representation theorem  could provide a numerical algorithm for the continuous-time problem.
	Here we consider the exit contract design problem, and provide a convergence result of the discrete-time problems to the continuous-time problem.

	\vspace{0.5em}

	Let us consider a sequence $(\pi_m)_{m \ge 1}$ of partitions of $[0,T]$, with $\pi_m = (t^m_j)_{0 \le j \le m}$,
	and such that $|\pi_m| := \max_{j=0, \cdots, m-1} (t^m_{j+1} - t^m_{j}) \longrightarrow 0$ as $m \longrightarrow \infty$.
	We also fix $\mu^{A, \pi_m}$ and $\mu^{P, \pi_m}$ by
	$$
		\mu^{A, \pi_m} (dt) := \sum_{j=1}^{m} c^{A, m} \delta_{t^m_j} (dt),
		~~\mbox{and}~
		\mu^{P, \pi_m} (dt) := \sum_{j=1}^{m} c^{P, m} \delta_{t^m_j} (dt),
	$$
	with
	$$
		c^{A, m}_j := \mu^A((t^m_{j -1}, t^m_{j}])
		~\mbox{and}~
		c^{P,m} := \mu^P ((t^m_{j -1}, t^m_{j}]).
	$$
	Recall also that the corresponding value function $V^{P, \pi_m}$ is defined by \eqref{eq:PbP_pi} with partition $\pi_m$.

	\begin{assumption}\label{Assum:ConvergenceDiscreteContinuous}
		$\mathrm{(i)}$
		For each $i = 1, \cdots, n$, one has
		$$
			\E\Big[\sup_{t\in [0,T]}|g_i(t)|\Big]
			+
			\E\Big[\sup_{t\in [0,T]}|f_i(t)|\Big]
			< +\infty.
		$$

		\noindent $\mathrm{(ii)}$
		For each $i = 1, \cdots, n$,
		\begin{align} \label{eq:int_fg_pim}
			& \lim_{m \to \infty}
			\sum_{j = 1}^{m}
			\int_{(t^{m}_{j - 1},t_{j}^{m}]}\big|
			f_i\big(t\big)
			- f_i(t_{j}^{m})\big|
			\mu^A(dt) = 0,
            ~\mbox{a.s.}\\
			~~~
			& \lim_{m \to \infty}
			\sum_{j = 1}^{m}
			\int_{(t^{m}_{j - 1},t_{j}^{m}]}\big|
    		        g_i\big(t\big)
			- g_i(t_{j}^{m})\big|
			\mu^P(dt) = 0,
			~\mbox{a.s.}
		\end{align}	
	\end{assumption}

	\begin{theorem}
		Let Assumptions \ref{assum:f} and \ref{Assum:ConvergenceDiscreteContinuous} hold true.
		Then
		\begin{equation}
			\lim_{m \to \infty} V^{P, \pi_m} = V^P.
		\end{equation}
	\end{theorem}
	\proof
	$\mathrm{(i)}$ For any fixed $L \in \Lc^+_0$, and $m \ge 1$, let us define
	$$
		L^{m}_t
		~ := ~
		\sum_{j = 0}^{m-1}
		L_{t_j^m}\mathds{1}_{[t_{j}^m,t_{j+1}^m)}(t) + L_T \mathds{1}_{\{T\}} (t),
		~t \in [0,T],
	$$
	so that $L^{m} \in \Lc^{\pi_m, +}_0$.
	Notice that $L^{m}$, $L$ are nondereasing and take only values in $\{0, 1, \cdots, n\}$,  then
	$$
		L^{m}_{t_j^m-} = L_{t_j^m-}
	$$
	holds except for at most $n$ numbers of time $t^m_j$.
	Therefore, for each $i=1, \cdots, n$,
	\begin{align*}
		& ~
		\int_{[0,T]}\Big|
			f\big(t,L^{m}_{t-}\big)
			\mathds{1}_{\{L^{m}_{t-} \geq i\}}
			-
			f\big(t,L_{t-}\big)
			\mathds{1}_{\{L_{t-} \geq i\}}\Big|
		\mu^{A,\pi_m}(dt) \\
		~= & ~
		\sum_{j = 0}^{m}\Big|
		f\big({t_{j}^{m}},L^{m}_{t_{j}^{m}-}\big)
		\mathds{1}_{\{L^{m}_{t_{j}^{m}-} \geq i\}}
		-
		f\big({t_{j}^{m}},L_{t_{j}^{m}-}\big)
		\mathds{1}_{\{L_{t_{j}^{m}-} \geq i\}}\Big|
		\mu^A((t_{j - 1}^{m}, t_{j}^{m}]) \\
		~ \leq & ~
		2n\Big(\sup_{t\in [0,T]}|f_1(t)| +
		\sup_{t\in [0,T]}|f_n(t)| + 1\Big)
		~\Big( \max_j\mu^A((t_{j - 1}^{m}, t_{j}^{m}]) \Big) .
	\end{align*}

	It follows that
	\begin{align} \label{eq:Discrete2ContinuousEstimation1}
		& ~
		\lim_{m \to \infty}
		\E\Big[\int_{[0,T]}\Big|
			f\big(t,L^{m}_{t-}\big)
			\mathds{1}_{\{L^{m}_{t-} \geq i\}}
			-
			f\big(t,L_{t-}\big)
			\mathds{1}_{\{L_{t-} \geq i\}}\Big|
			\mu^{A,\pi_m}(dt)\Big] \nonumber \\
		~ \leq & ~
		\lim_{m \to \infty}
		2n\E\Big[\sup_{t\in [0,T]}|f_1(t)| +
			\sup_{t\in [0,T]}|f_n(t)| + 1\Big]
			\Big( \max_j\mu^A((t_{j - 1}^{m}, t_{j}^{m}]) \Big)
		~=~ 0.
	\end{align}
	Further, for $L \in \Lc^+_0$, $L_{t-}$ remains constant on $(t^{m}_{j - 1},t_{j}^{m}]$ except for at most $n$ numbers of time $t^m_j$.
	It follows that, for each $i=1, \cdots, n$,
	\begin{align*}
		& ~
		\Big|\E\Big[\int_{[0,T]}
			f\big(t,L_{t-}\big)
			\mathds{1}_{\{L_{t-} \geq i\}}
		\Big(\mu^{A,\pi_m}(dt)- \mu^A(dt)\Big)\Big]\Big|
		\\
		~ = & ~
        \Big|\E\Big[\sum_{j = 1}^{m}
            \int_{(t^{m}_{j - 1},t_{j}^{m}]}\Big(
            f\big(t,L_{t-}\big)
                - f\big(t^{m}_i,L_{t_{j}^{m}-}\big)\Big)
                \mathds{1}_{\{L_{t-} \geq i\}}
             \mu^A(dt)\Big]\Big|
		\\
	~ \leq & ~
        2n(M + 1) \Big( \max_j\mu^A((t_{j - 1}^{m}, t_{j}^{m}]) \Big)
        + \sum_{i = 0}^{n}\E\Big[\sum_{j = 1}^{m}
            \int_{(t^{m}_{j - 1},t_{j}^{m}]}\Big|
            f_i\big(t\big)
                - f_i\big(t^{m}_j\big)\Big|
            \mu^A(dt)\Big],
	\end{align*}
	with
	$$
		M ~:=~
		\E\Big[\sup_{t\in [0,T]}|f_1(t)| +
		\sup_{t\in [0,T]}|f_n(t)| \Big].
	$$
	By \eqref{eq:int_fg_pim} and dominated convergence theorem,
        we have for $i = 0, 1, \cdots, n$
	$$
        \lim_{m \to \infty}
            \E\Big[\sum_{j = 1}^{m}
            \int_{(t^{m}_{j - 1},t_{j}^{m}]}\Big|
            f_i\big(t\big)
                - f_i(t_{j}^{m})\Big|
            \mu^A(dt)\Big] = 0,
	$$
	and it follows that
	\begin{equation}\label{eq:Discrete2ContinuousEstimation2}
		\lim_{m \to \infty}
		\E\Big[\int_{[0,T]}
			f\big(t,L_{t-}\big)
			\mathds{1}_{\{L_{t-} \geq i\}}
		\Big(\mu^{A,\pi_m}(dt)- \mu^A(dt)\Big)\Big]
		~ = ~
		0.
	\end{equation}
	Then by \eqref{eq:Discrete2ContinuousEstimation1} and \eqref{eq:Discrete2ContinuousEstimation2}, one obtains that
   \begin{align*}
        \lim_{m \to \infty}
        \E\Big[\int_{[0,T]}
            f\big(t,L^{m}_{t-}\big)
            \mathds{1}_{\{L^{m}_{t-} \geq i\}}
             \mu^{A, \pi_m}(dt)\Big]
        ~  =  ~
        \E\Big[\int_{[0,T]}
            f\big(t,L_{t-}\big)
            \mathds{1}_{\{L_{t-} \geq i\}}
             \mu^A(dt)\Big].
   \end{align*}
	Similarly, one can obtain that
	$$
		\lim_{m \to \infty}\E\Bigg[
		\int_{[0,T]}
		g_i\big(t\big)
		\mathds{1}_{\{L^{m}_{t-} \geq i\}}\mu^{P, \pi_m}(dt)
		\Bigg]
		~ = ~
		\E\Bigg[\int_{[0,T]}
		g_i\big(t\big)
		\mathds{1}_{\{L_{t-} \geq i\}}\mu^P(dt)
		\Bigg].
	$$
	This implies that
	\begin{align*}
		& ~
		\lim_{m \to \infty}
		\E\Bigg[\sum_{i=1}^{n}
		\int_{[0,T]}\bigg(
		g_i(t)
                \mathds{1}_{\{L^{m}_{t-} < i\}}\mu^{P,\pi_m}(dt) -
		f\big(t,L^{m}_{t-}\big)
                \mathds{1}_{\{L^{m}_{t-} \geq i\}}\mu^{A,\pi_m}(dt)
		\bigg) - \xi\Bigg]
		\\
		~ = & ~
		\E\Bigg[\sum_{i=1}^{n}
		\int_{[0,T]}\bigg(
		g_i(t)
		\mathds{1}_{\{L_{t-} < i\}}\mu^P(dt) -
		f\big(t,L_{t-}\big)
                \mathds{1}_{\{L_{t-} \geq i\}}\mu^A(dt)
		\bigg) - \xi\Bigg].
	\end{align*}
	Further, by the arbitrariness of $L$, this leads to the inequality
	\begin{equation*}
		\liminf_{m \to \infty}V^{P, \pi_m} ~\geq~ V^P.
	\end{equation*}

	\noindent $\mathrm{(ii)}$ To prove the reverse inequality,
	we notice that $L \in \Lc^+_0$ for any $L \in \Lc^{\pi_m, +}_0$.
	Then, for each $i=1, \cdots, n$, one has the estimation
	$$
		\Big|\E\Big[\int_{[0,T]}
			f\big(t,L_{t-}\big)
			\mathds{1}_{\{L_{t-} \geq i\}}
			\Big(\mu^{A,\pi_m}(dt)- \mu^A(dt)\Big)\Big]\Big|
			\\
		~ \leq~
		\eps_m,
	$$
	where $\eps_m$ is independent of $L$, and satisfies
	$$
		\eps_m
		~=~
		2n(M + 1) \Big( \max_j\mu^A((t_{j - 1}^{m}, t_{j}^{m}]) \Big)
		+ \sum_{i = 0}^{n}\E\Big[\sum_{j = 1}^{m}
		\int_{(t^{m}_{j - 1},t_{j}^{m}]}\Big|
		f_i\big(t\big)
		- f_i\big(t^{m}_j\big)\Big|
		\mu^A(dt)\Big]
		~\longrightarrow~ 0,
	$$
	as $m \longrightarrow \infty$.
	Similarly, one has, for any $L \in \Lc^{\pi_m,+}_0$
	$$
		\eps_m' ~:=~
		\Big|\E\Big[\int_{[0,T]}
		g_i\big(t\big)
		\mathds{1}_{\{L_{t-} < i\}}
		\Big(\mu^{A,\pi_m}(dt)- \mu^P(dt)\Big)\Big]\Big|
		~\longrightarrow~ 0,
		~\mbox{as}~
		m \longrightarrow \infty.
	$$
	Then it follows that
	$$
		V^{P, \pi_m} ~\le~ V^P + \eps_m + \eps_m',
	$$
	which concludes the proof.
\endproof

\section{The problem with Markovian and/or continuous contract}
\label{sec:Markovian}

	In this section, we will investigate a version of the exit contract design problem, where the admissible contracts are required to be Markovian and/or continuous w.r.t. some underlying process $X$.
	We will restrict ourself in the following discrete-time setting with a partition $\pi=(t_j)_{0 \leq j \leq m}$ of $[0,T]$, i.e. $0 = t_0 < \cdots < t_m = T$,
	$\mu^A(dt) := \sum_{j = 1}^mc_j^A\delta_{t_j}(dt)$ and $\mu^P(dt) := \sum_{j = 1}^mc_j^P\delta_{t_j}(dt)$,  with constants $c^A_j$, $c^P_j > 0$ for all $j = 1, \cdots, m$.

	\vspace{0.5em}

	Let $\Om = \R^{d \x (m+1)}$ be the canonical space with canonical process $X = (X_{t_j})_{1 \le j \le m}$, i.e. $X_{t_j} (\om) := \om_j$ for all $\om = (\om_0, \cdots, \om_m) \in \Om$.
	We further extend it to be a continuous-time process by setting $X_t = X_{t_j}$, for all $t \in [t_j, t_{j+1})$, $j=0, \cdots, m-1$.
	By abus of notation, we still denote it by $X = (X_t)_{t \in [0,T]}$.
	Let $\Fc := \Fc_T$, with $\Fc_t := \sigma(X_s ~: s  \in [0,t])$ for every $t \in [0,T]$, so that the filtration $\F = (\Fc_t)_{t \in [0,T]}$ is right-continuous and $\F^{\pi} = \F$ by \eqref{eq:def_Fpi}.
	We next assume that $\P$ is a probability measure on $(\Om, \Fc)$, under which $(X_{t_j})_{j = 0, \cdots, m}$ is a Markovian process,
	i.e.
	$$
		\Lc^{\P} \big( (X_{t_{j+1}}, \cdots, X_{t_m}) \big| X_0, \cdots, X_{t_j} \big)
		=
		\Lc^{\P} \big( (X_{t_{j+1}}, \cdots, X_{t_m}) \big| X_{t_j} \big)
		~\mbox{a.s. for each}~
		j = 0, \cdots, m-1.
	$$
	Then there exists also a family of probability measures $\{ \P^j_x ~:j = 0, \cdots, m, x \in \R^d \}$ such that $(\P^j_x)_{x \in \R^d}$ consists of a family of conditional probability distribution of $(X_{t_j}, \cdots, X_{t_m})$ known $X_{t_j}$.	

	\vspace{0.5em}

	Since $\F$ is generated by $X$, an admissible contract $Y \in \Yc^{\pi}$ (see its definition in \eqref{eq:def_Ycpi}) is a functional of process $X$.
	We will further define a class of Markovian contracts in the sense that $Y_{t_j} = y_j(X_{t_j})$ for some function $y_j$,
	and a class of continuous Markovian contracts by assuming $y$ to be continuous.
	Let us denote by $B(\R^d)$ the collection of all Borel measurable functions defined on $\R^d$, and by $C_b(\R^d)$ the collection of all bounded continuous functions defined on $\R^d$.
	Let
	\begin{align*}
		\Yc^\pi_{m}
		~:=~ &
		\big\{
			Y \in \Yc^\pi~:
			Y_{t_j} = y_j (X_{t_j}),
			~\mbox{for some}~y_j \in B(\R^d),
			~j = 0, 1, \cdots, m
		\big\},
		\\
		\Yc^\pi_{m,c}
		~:=~ &
		\big\{
			Y \in \Yc^\pi~:
			Y_{t_j} = y_j (X_{t_j}),
			~\mbox{for some}~y_j \in C_b(\R^d),
			~j = 0, 1, \cdots, m
		\big\}.
	\end{align*}
	We then obtain two variations of the contract design problem:
	\begin{align} \label{eq:PbP_pim}
 		V^{P,\pi}_m
		~ & :=~
		\sup_{Y \in \Yc^{\pi}_m} \E \bigg[ \sum_{i=1}^n \Big( \int_{[0,\tauh_i]} g_i(t) \mu^{P,\pi}(dt) - Y_{\tauh_i} \Big) \bigg],
		\\ \label{eq:PbP_pimc}
		V^{P,\pi}_{m,c}
		~ & :=~
		\sup_{Y \in \Yc^{\pi}_{m,c}} \E \bigg[ \sum_{i=1}^n \Big( \int_{[0,\tauh_i]} g_i(t) \mu^{P,\pi}(dt) - Y_{\tauh_i} \Big) \bigg],
	\end{align}
	where $\tauh_i$ is the minimum optimal stopping time of the agent $i$ (with the corresponding given contract $Y$) in \eqref{eq:VApi}.

	\vspace{0.5em}

	We can naturally expect to solve the above contract design problem as in Theorem \ref{thm:main_discrete}, but consider a subset of $\Lc^{\pi}$.
	Let us define
	\begin{align*}
		\Lc^\pi_{m}
		~:=~ &
		\big\{
			L \in \Lc^\pi~:
			L_{t_j} = l_j(X_{t_j})
			~\mbox{for some}~
			l_j \in B(\R^d), ~j = 0, 1, \cdots, m
		\big\}, \\
		\Lc^\pi_{m,c}
		~:=~ &
		\big\{
			L \in \Lc^\pi~:
			L_{t_j} = l_j(X_{t_j})
			~\mbox{for some}~
			l_j \in C_b(\R^d), ~j = 0, 1, \cdots, m
		\big\},
	\end{align*}
	and with $\Lc^\pi$, $\Lc^{\pi, +}$ and $\Lc^{\pi, +}_0$ defined above and in \eqref{eq:def_Lc_pi}, we let
	\begin{equation*}
		\Lc^{\pi,+}_{m}:= \Lc^{\pi}_{m} \cap \Lc^{+}, ~~
		\Lc^{\pi,+}_{m,0}:= \Lc^{\pi}_{m} \cap \Lc^{+}_0, ~~
		\Lc^{\pi,+}_{m,c}:= \Lc^{\pi}_{m,c} \cap \Lc^{+}.
	\end{equation*}

	We next formulate some additional technical conditions on the coefficient functions $\xi$, $f$ and probability kernels $(\P^j_x)_{x \in \R^d}$.

	\begin{assumption} \label{assum:Markovian}
		For each $i=1, \cdots, n$ and $t \in [0,T]$, the random variables $f_i(t, \cdot)$ and $g_i(t, \cdot)$ satisfies
		$\E \big[ |f_i(t)| + |g_i(t)| \big] < \infty$, and $f_i(t, \cdot)$ depends only on $X_{t}$.
		The random variable $\xi$ depends only on $X_T$.
		Moreover, one has
		$$
			f_1(t, \om) < \cdots < f_n(t, \om),  ~~\mbox{for all}~(t, \om) \in [0,T] \x \Om.
		$$
	\end{assumption}

	Recall that, in \eqref{eq:f_interpolation}, the functionals $(f_1, \cdots, f_n)$ would be interpolated into a functional $f: [0,T] \x \Om \x \R \longrightarrow \R$.
	Notice also that $g_i(t, \cdot)$ is only $\Fc_t$--measurable, it could depends on the whole past path of $X_{t \wedge \cdot}$.
	To emphasize this point, we will write
	$$
		\big( f_i(t, X_t),~f(t, X_t, \ell), ~ g(t, X_{t \wedge \cdot}),  \xi(X_T) \big)
		~~\mbox{in place of}~
		\big( f_i(t, \om),~f(t, \om, \ell), ~ g(t, \om), \xi(\om) \big).
	$$

	\begin{assumption}\label{Assume:ContinuousLaw}
		For each $i=1, \cdots, n$ and $j=0, \cdots, m-1$, the function $f_i(t_j, \cdot) \in C_b(\R^d)$,
		and the map $x \longmapsto \P_x^j$ is continuous under the weak convergence topology.
	\end{assumption}

	\begin{example}
		Let $\Xb$ be a diffusion process defined by the stochastic differential equation
		$$
			d \Xb_t ~=~ \mu(\Xb_t) dt ~+~ \sigma(\Xb_t) dW_t,
		$$
		with a Brownian motion $W$  in the filtered probability space $(\Om, \Fc, \F, \P)$ and the Lipschitz coefficient functions $b$ and $\sigma$.
		Let $X_t = \Xb_{t_j}$ for all $t \in [t_j, t_{j+1})$ and $j=0, \cdots, m-1$.
		Then it satisfies Assumption \ref{Assume:ContinuousLaw}.
	\end{example}

	\begin{theorem}\label{thm:MarkovContinuousContract}
		Let Assumption \ref{assum:Markovian} hold true.
		Then
		\begin{align}\label{equiv:PrincipalPb-markov}
			V^{P,\pi}_m
			\!=\!\!
			\sup_{L \in \Lc^{\pi,+}_{m,0}} \!
			\E\bigg[ \! \sum_{i=1}^{n}  \!\! \bigg( \!\!
			\int_{[0,T]} \!\!\! \Big(g_i(t, X_{t\wedge \cdot})\mathds{1}_{\{L_{t-} < i\}}\mu^{P, \pi} (dt)
				\!-\!
				f\big(t, X_t,L_{t-}\big)\mathds{1}_{\{L_{t-} \geq i\}}\mu^{A, \pi} (dt) \! \Big)
				\!-\!
				\xi(X_T) \!\! \bigg) \! \bigg].
		\end{align}
		Let Assumptions \ref{assum:Markovian}  and \ref{Assume:ContinuousLaw} hold true.
		Then
		\begin{align} \label{equiv:PrincipalPb-continuous}
			V^{P,\pi}_{m,c}
			\!=\!\!
			\sup_{L \in \Lc^{\pi,+}_{m,c}} \!\!
			\E\bigg[ \! \sum_{i=1}^{n}  \!\! \bigg( \!\!
			\int_{[0,T]} \!\!\! \Big(g_i(t, X_{t\wedge \cdot})\mathds{1}_{\{L_{t-} < i\}}\mu^{P, \pi} (dt)
				\!-\!
				f\big(t, X_t,L_{t-}\big)\mathds{1}_{\{L_{t-} \geq i\}}\mu^{A, \pi} (dt) \! \Big)
				\!-\!
				\xi(X_T) \!\! \bigg) \! \bigg].
		\end{align}
	\end{theorem}

	\begin{remark} \label{rem:no_stopping_formulation}
		Similarly to Theorem \ref{ReductionByBErepresentation} and Proposition \ref{prop:P_SnellEnv},
		it is equivalent to take the supremum over $\Lc^{\pi,+}_m$ in place of $\Lc^{\pi,+}_{m,0}$
		{at the r.h.s. of \eqref{equiv:PrincipalPb-markov} (reps. \eqref{equiv:PrincipalPb-continuous}).}
		Moreover, one can construct {the corresponding optimal solutions} from the initial contract design problem in \eqref{eq:PbP_pim} (resp. \eqref{eq:PbP_pimc}), and vice versa.

		\vspace{0.5em}

		Nevertheless, we do not have now the equivalent optimal stopping formulation for \eqref{equiv:PrincipalPb-markov} and \eqref{equiv:PrincipalPb-continuous} as in Proposition \ref{prop:P_SnellEnv}.
		Intuitively, given a solution $\tauh_1 \le \cdots \tauh_n$ to the optimal stopping problem as at the r.h.s. of \ref{prop:P_SnellEnv},
		the corresponding $\widehat L^*_t := \sum_{i=1}^n \mathds{1}_{\{\tauh_i < t\}}$ is a priori measurable w.r.t. $\Fc_t = \sigma(X_s, s \le t)$, but not measurable w.r.t. $\sigma(X_t)$.
		Therefore, at time $t \in [0,T]$, the contract value $Y^{\widehat L^*}_t$ defined in \eqref{YdefinedL} has no reason to be a measurable function of $X_t$.
	\end{remark}

	As preparation of the above theorem, let us provide some technical lemmas.

	\begin{lemma}\label{SnellenvelopMarkovian}
		Let Assumption \ref{assum:Markovian} hold true.
		Then for each $Y \in \Yc^{\pi}_m$ and $j=0, \cdots, m$, the random variable $Z^\ell_{t_j}$ defined below is $\sigma(X_{t_j})$-measurable (up to the complementation of the $\sigma$-field):
		\begin{equation} \label{eq:def_Zltj}
			Z^{\ell}_{t_j}
			~:=~
			\esssup_{\tau \in \Tc_{t_j}}
			\E \Big[ \int_{(t_j, {\tau}]} f(s,X_s, \ell) \mu^{A, \pi} (ds)
				+ Y_{\tau} \Big|\Fc_{t_j} \Big].
		\end{equation}
	\end{lemma}
	\begin{proof}
	Recall that $\mu^{A, \pi}$ is sum of Dirac measures on $\pi$, then by the dynamic programming principle, one has $Z^\ell_{t_m} = Y_{t_m}$ a.s., and
	$$
		Z^\ell_{t_j}
		=
		\max \Big\{
			\E \big[Z^\ell_{t_{j + 1}} + c^{A}_{j + 1}f(t_{j + 1},X_{t_{j + 1}},\ell) \big|\Fc_{t_j} \big],
			~Y_{t_{j}}
		\Big\}
		~\mbox{a.s.}, ~j=0, \cdots, m-1.
	$$
	It is then enough to apply the induction argument.
	First, $Z^\ell_{t_m} = Y_{t_m}$ is $\sigma(X_{t_m})$-measurable as $t_m = T$.
	Next, assume that that $Z^\ell_{t_{j + 1}}$ is $\sigma(X_{t_{j + 1}})$ measurable,
	then by the Markov property of $X$, one has
	$$
		\E \big[ Z^\ell_{t_{j + 1}} + c^{A}_{j + 1}f(t_{j + 1},X_{t_{j + 1}},\ell) \big| \Fc_{t_{j}} \big]
		~\mbox{is}~
		\sigma(X_{t_j}) \mbox{-measurable}.
	$$
	Thus $Z^\ell_{t_{j}}$ is also $\sigma(X_{t_{j}})$-measurable.
	\end{proof}

	\vspace{0.5em}
	
	Next, given $Y \in \Yc^{\pi}_m$, let us define
	\begin{equation} \label{eq:def_LY_pi}
		L^Y_{t_j}
		~:=~
		\sup \big\{\ell:Z^\ell_{t_j} = Y_{t_j} \big\};
	\end{equation}
	and given $L \in  \Lc^{\pi,+}_m$, we define
	\begin{equation}\label{eq:LdefinedYdiscreteMarkov}
		Y^L_{s}
		~:=~
		\E \Big[\xi(X_T) + \int_{(s,T]}
		f\Big(t, X_t, L_{t-} \Big)\mu^{A, \pi} (dt) \Big| \Fc_s\Big],
		~\mbox{a.s.}
	\end{equation}
	
	\begin{lemma}\label{Equiv:LandYmeasurable}
		Let Assumption \ref{assum:Markovian} hold true.
		Then given $Y \in \Yc^{\pi}_{m}$, one has $L^Y \in \Lc^{\pi}_m$.
		Moreover, given $L \in \Lc^{\pi,+}_m$, one has $Y^L \in \Yc^{\pi}_{m}$.
	\end{lemma}
	\begin{proof}
	$\mathrm{(i)}$ Given $Y \in \Yc^{\pi}_{m}$, it is easy to observe (see also the proof of Theorem \ref{thm:BErepresentationTheorem}) that
	$\ell \longmapsto Z^\ell_{t_j}$ is continuous a.s.
	It follows that
	\begin{equation*}
		\{\ell > L^Y_{t_j}\}
		=
		\{Z^\ell_{t_j} > Y_{t_j}\}
		=
		\bigcup_{\substack{r_m,r_k \in \Q\\ r_k < r_m}}
		\{Z^\ell_{t_j} > r_m\} \cap \{r_k > Y_{t_j}\}.
	\end{equation*}
	This is enough to prove that $L^Y_{t_j}$ is $\sigma(X_{t_j})$-measurable as $Y_{t_j}$ and $Z^{\ell}_{t_j}$ (for all $\ell \in \R$) are all $\sigma(X_{t_j})$-measurable.
	Therefore, $L^Y \in \Lc^\pi$.	
	
	\vspace{0.5em}
	
	\noindent $\mathrm{(ii)}$ Given $L \in \Lc^{\pi,+}_m$, it is easy to check that $Y^L$ is piecewise constant on each interval $[t_j, t_{j+1})$.
	Moreover, for  $t_j = s < t \le t_{j+1}$, one has $\sup_{v \in [0,t)} L_v = L_{t_j}$.
	Then it is enough to apply the Markovian property of $X$ to prove that $Y^L_{t_j}$ is $\sigma(X_{t_j})$-measurable.
	\end{proof}

	\begin{lemma}\label{Lemma:ConditionalContinuity}
		Let Assumptions \ref{assum:Markovian} and \ref{Assume:ContinuousLaw} hold true.
		Let $h:\R^d \times \R \to \R$ be such that
		$x \longmapsto h(x, \ell)$ is bounded continuous for all $\ell \in \R$ and
		$\ell \longmapsto h(x, \ell)$ is Lipschitz continuous uniformly in $x \in \R^d$.
		For a fixed $ j \in \{0, \cdots, m-1\}$, let us define $\hat{h}:\R^d \times \R \to \R$ by
		$$
			\hat{h}(x,\ell) ~:=~ \E \big[ h(X_{{t_{j + 1}}}, \ell) \big | X_{t_j} = x \big].
		$$
		Then $\hat h(x, \ell)$ is also bounded continuous in $x$, and Lipschitz in $\ell$ uniformly in $x$.

		\vspace{0.5em}
		
		Assume in addition that $\ell \longmapsto h(x, \ell)$ is strictly increasing and $\lim_{\ell \to \pm \infty}h(x, \ell) = \pm \infty$,
		then $\ell \longmapsto \hat{h}(x, \ell)$ is also strictly increasing and $\lim_{\ell \to \pm \infty}\hat{h}(x, \ell)  = \pm \infty$.
	\end{lemma}
	\begin{proof}
	\noindent $\mathrm{(i)}$
	Let us fixe $\ell \in \R$, $x \in \R^d$ together with a sequence $\{x_k\}_{k \ge 1} \subset \R^d$ s.t. $\lim_{k \to \infty}x_k = x$.
	By the continuity of $x \longmapsto \P_x^j$, it follows that
	\begin{align*}
		&
		\lim_{k \to \infty} \big| \hat{h}(x, \ell) - \hat{h}(x_k, \ell) \big|
		~=~
		\lim_{k \to \infty} \Big| \E^{\P_{x}^j} \big[h(X_{t_{j + 1}},\ell) \big] - \E^{\P_{x_k}^j} \big[h(X_{t_{j + 1}},\ell) \big] \Big|
		~=~
		0.
	\end{align*}
	Therefore, the function $x \longmapsto \hat{h}(x, \ell)$ is bounded continuous.
	Moreover, it is clear that $\ell \longmapsto \hat h(x, \ell) = \E^{\P_{x}^j} \big[h(X_{t_{j + 1}},\ell) \big] $ is also Lipschitz with the same Lipschitz constant of $x \longmapsto h(x, \ell)$.
	
	\vspace{0.5em}

	\noindent $\mathrm{(ii)}$ Let $x \in \R^d$ and $\ell_2 >  \ell_1$, one has
	\begin{align*}
		\hat{h}(x, \ell_2) - \hat{h}(x, \ell_1)
		~=~
		\E^{\P_{x}^j}[h(X_{t_{j + 1}},\ell_2) - h(X_{t_{j + 1}},\ell_1)]
		~>~
		0,
	\end{align*}
	so that $\ell \longmapsto \hat{h}(x, \ell)$ is strictly increasing.
	Finally, by the monotone convergence theorem, one obtains  $\lim_{\ell \to \pm \infty}\hat{h}(x, \ell) = \pm \infty$.
	\end{proof}

	\begin{lemma}\label{Equiv:LandYcontinuous}
		Let Assumptions \ref{assum:Markovian} and \ref{Assume:ContinuousLaw} hold true.

		\vspace{0.5em}

		\noindent $\mathrm{(i)}$
		For any $Y \in \Yc^\pi_{m,c}$, there exist mappings $l_j (\cdot) \in C_b(\R^d)$ such that for $j = 0, 1, \cdots, m$,
		$$
			L^Y_{t_j}  ~=~ l_j( X_{t_j}),
		$$
		where $L^Y_{t_j}$ is defined by \eqref{eq:def_LY_pi}, so that $L^Y \in \Lc^\pi_{m, c}$.

		\vspace{0.5em}

		\noindent $\mathrm{(ii)}$
		For any $L \in \Lc^{\pi,+}_{m,c}$, there exist mappings $ y_j (\cdot)\in C_b(\R^d) $ such that for $j = 0, 1, \cdots, m$,
		$$
			Y^L_{t_j} ~=~ y_j(X_{t_j}),
		$$
		where $Y^L$ is defined by \eqref{eq:LdefinedYdiscreteMarkov}, so that $Y^L \in \Yc^\pi_{m, c}$.
	\end{lemma}
	\proof
	$\mathrm{(i)}$ Let $Y \in \Yc^\pi_{m, c}$.
	By definition, there exist bounded continuous mappings $y_j (\cdot):\R^d \longrightarrow \R$, such that $Y(t_j) = y_j (X_{t_j})$ for all $j = 0, \cdots, m$.
	Recall that $Z^{\ell}_{t_j}$ is defined in \eqref{eq:def_Zltj} and satisfies
	$$
		Z^\ell_{t_j}
		=
		\max \Big\{
			\E \big[Z^\ell_{t_{j + 1}} + c^{A}_{j + 1}f(t_{j + 1},X_{t_{j + 1}},\ell) \big|\Fc_{t_j} \big],
			~Y_{t_{j}}
		\Big\}
		~\mbox{a.s.}, ~j=0, \cdots, m-1.
	$$
	Then by Lemma \ref{Lemma:ConditionalContinuity}, together with a backward induction argument,
	there exist bounded continuous functions $z^\ell_j : \R^d \longrightarrow \R$ such that
	$$
		Z^\ell_{t_j}
		=
		z^\ell_j(X_{t_j}),
		~\mbox{for all}~
		j = 0, \cdots, m.
	$$
	
	To conclude the proof of $\mathrm{(i)}$, it is enough to prove that
	$$
		x
		~\longmapsto~
		l^Y_j (x)
		:=
		\sup \big\{\ell ~:z^\ell_j (x) = y_j (\cdot) \big\}
		\in C_b(\R^d).
        $$
        Let us define $\hat z^{\ell}_{j+1}$ and $\hat f(t_{j+1}, \cdot)$ by
        $$
		\hat z^{\ell}_{j+1}(x) ~:=~ \E \big[ z^{\ell}_{j+1}(X_{t_{j+1}}) \big| X_{t_j} = x \big],
		~\mbox{and}~
		\hat f(t_{j+1}, x, \ell) ~:=~  \E \big[ f(t_{j+1}, X_{t_{j+1}}, \ell) \big| X_{t_j} = x \big].
        $$
	By Lemma \ref{Lemma:ConditionalContinuity} again, the maps $(x, \ell) \longmapsto \hat z^{\ell}_{j+1}(x)$ and $(x, \ell) \longmapsto \hat f(t_{j+1}, x, \ell))$ are continuous,
	and strictly increasing in $\ell$.
	Moreover, it is easy to check that
	$$
		\lim_{l \to \pm\infty}\hat{f}({t_{j + 1}}, x, \ell)
		=
		\lim_{l \to \pm\infty}\hat{z}^\ell_{j + 1} (x)
		= \pm \infty.
	$$
	Therefore, $l^Y_j(x)$ is the unique $\ell \in \R$ such that
	\begin{equation}\label{Equation:LMarkov}
		\hat{z}^\ell_{j + 1} (x) + \hat{f} (t_{j + 1}, x,\ell)
		~=~
		y_j (x).
	\end{equation}
	Recall that  $ \ell \longmapsto \hat z^{\ell}_{j+1}(x)$ and $ \ell \longmapsto \hat f(t_{j+1}, x, \ell))$ are Lipschitz, uniformly in $x$, and $x \longmapsto l^Y_j(x)$ is bounded continuous.
	Let $(x_k)_{k \ge 1} \subset \R^d$ be such that $\lim_{k \to \infty}x_k = x \in \R^d$,
	then the sequence of real numbers $\{{l^Y_j} (x_k)\}_{k = 1}^\infty$ is uniformly bounded and
	\begin{align*}
		y_j (x)
		~ = ~
		\hat{z}^{\limsup_{k \to \infty}{l^Y_j} (x_k)}_{j + 1} (x) + \hat{f} \big(t_{j + 1},x, \limsup_{k \to \infty} l^Y_j (x_k) \big).
	\end{align*}
	Therefore
	$$
		\limsup_{k \to \infty} l^Y_j (x_k)
		=
		l^Y_j (x).
	$$
	Similarly, we can prove
	$$
		\liminf_{k \to \infty} l^Y_j (x_k)
		=
		l^Y_j (x),
		~~\mbox{and hence}~
		\lim_{k \to \infty} l^Y_j (x_k)
		=
		l^Y_j (x).
	$$

	\noindent $\mathrm{(ii)}$
	Let $L \in \Lc^{\pi, +}_{m,c,}$ and $j = 0, 1, \cdots, m$, notice that $x \longmapsto f(t_k,x,\ell)$ is bounded continuous  for each $k > j$, and $x \longmapsto l({t_k},x)$ is bounded continuous for each $k \geq j$.
	It is enough to apply Lemma \ref{Lemma:ConditionalContinuity} to show that $Y^L$ is a bounded continuous function of $X_{t_j}$.
	\endproof

	\vspace{0.5em}

	\proof \textbf{of Theorem \ref{thm:MarkovContinuousContract}.}
	The idea and procedure are the same as Theorem \ref{ReductionByBErepresentation}. Only subtle details differ, so we only point out the differences in the proof and omit the others.

	\vspace{0.5em}

	\noindent $\mathrm{(i)}$
	Let us first claim that
	\begin{equation}\label{Inequality:reduction-leqMarkov}
		V^{P,\pi}_m
		\leq
		\sup_{L \in \Lc^{\pi,+}_m}
		\E\bigg[\sum_{i=1}^{n} \bigg(
		\int_{[0,T]}\Big(g_i(t, X_{t\wedge \cdot})\mathds{1}_{\{L_{t-} < i\}}\mu^{P, \pi} (dt) - f\big(t, X_t, L_{t-}\big)\mathds{1}_{\{L_{t-} \geq i\}}\mu^{A, \pi} (dt) \Big) -  \xi \bigg) \bigg].
	\end{equation}
	Given any $Y\in \Yc^\pi_m$, we define by $L^Y$ by \eqref{eq:def_LY_pi}, so that
	$$
		Y_\tau
        		=
        		\E\Big[\int_{(\tau,T]}
        		f\Big(t,X_t, \sup_{s \in [\tau,t)}L_s \Big)
        		\mu^{A, \pi}(dt)
        		+ \xi \Big| \Fc^{X}_\tau\Big],
		~\mbox{a.s., for all}~
		\tau \in \Tc.
	$$
	Let
	$$
		\Lh^{Y}_t
		~:=~
		0 \vee \sup_{s\in [0,t)}L^Y_s \wedge n,
		~~~ t\in [0,T].
	$$
	By similar procedure in Theorem \ref{ReductionByBErepresentation}, it follows that for each $i = 1, \cdots, n$
	\begin{equation*}
		J^P_i(Y)
		~ \leq ~
		\E \Big[
			\int_{[0,T]} \Big( g_i(t, X_{t\wedge \cdot}) \mathds{1}_{\{\Lh^{Y}_t < i\}}\mu^{P,\pi}(dt) - f\big(t, X_t, \Lh^{Y}_t \big)\mathds{1}_{\{\Lh^{Y}_t \geq i\}}  \mu^{A,\pi}(dt)\Big) - \xi
		\Big].
	\end{equation*}
	Notice that Lemma \ref{Equiv:LandYmeasurable} implies that $L^Y \in \Lc^\pi_m$, then it is easy to verify that $\Lh^Y \in \Lc^{\pi,+}_m$.
	Finally, taking the sum over $i=1, \cdots, n$,
	we prove the inequality \eqref{Inequality:reduction-leqMarkov}.
	
	\vspace{0.5em}
	
	\noindent $\mathrm{(ii)}$
	We next prove the reverse inequality:
	\begin{equation}\label{Inequality:reduction-geqMarkov}
		V^{P,\pi}_m
		~\geq
		\sup_{L \in \Lc^{\pi,+}_m}\!
		\E\bigg[ \sum_{i=1}^{n}  \!\bigg(\!
		\int_{[0,T]} \!\!\Big(g_i(t, X_{t\wedge \cdot})\mathds{1}_{\{L_{t-} < i\}}\mu^{P, \pi} (dt) - f\big(t, X_t, L_{t-}\big)\mathds{1}_{\{L_{t-} \geq i\}}\mu^{A, \pi} (dt) \Big) -  \xi \bigg) \bigg].
	\end{equation}

	For each $L \in \Lc^{\pi, +}_m$, let $Y^L$ be defined by \eqref{eq:LdefinedYdiscreteMarkov}.
	Then Lemma \ref{Equiv:LandYmeasurable} gives that $Y^{L} \in \Yc^{\pi}_m$.
	So we can derive as in Theorem \ref{ReductionByBErepresentation} similarly the inequality
	\begin{align*}
		V^P_m
		~ & \geq ~
		\sum_{i=1}^n J^P_i(Y^L)
		\\
		~ & = ~
		\E\bigg[\sum_{i=1}^{n} \bigg(
		\int_{[0,T]}\Big(g_i(t, X_{t\wedge \cdot})\mathds{1}_{\{L_{t-} < i\}}\mu^{P,\pi}(dt) - f\big(t, X_t,L_{t-}\big)\mathds{1}_{\{L_{t-} \geq i\}}\mu^{A,\pi}(dt) \Big) - \xi \bigg) \bigg],
	\end{align*}
	and therefore the reverse inequality \eqref{Inequality:reduction-geqMarkov} holds.

	\vspace{0.5em}

	\noindent $\mathrm{(iii)}$
	For any $L \in \Lc^{\pi, +}_m$, let us define $L^0_t := [L_t]$.
	Then $L^0 \in \Lc^{\pi,+}_{m, 0}$, $L^0 \leq L$ and $\{L_{t-} \geq i\} = \{L^0_{t-}\geq i\}$ for any $t\in [0,T]$, $i = 1, \cdots, n$.
	Hence we have
	\begin{align*}
		&\E\bigg[\sum_{i=1}^{n} \bigg(
		\int_{[0,T]}\Big(g_i(t, X_t)\mathds{1}_{\{L_{t-} < i\}}\mu^{P,\pi}(dt) - f\big(t, X_t, L_{t-}\big)\mathds{1}_{\{L_{t-} \geq i\}}\mu^{A,\pi}(dt)\Big) - \xi \bigg)\bigg] \\
		\leq~ &\E\bigg[\sum_{i=1}^{n} \bigg(
		\int_{[0,T]}\Big(g_i(t, X_t)\mathds{1}_{\{L^0_{t-} < i\}}\mu^{P,\pi}(dt) - f\big(t, X_t,L^0_{t-}\big)\mathds{1}_{\{L^0_{t-} \geq i\}}\mu^{A,\pi}(dt) \Big) - \xi \bigg)\bigg].
	\end{align*}
	So we complete the proof of  $\mathrm{(i)}$ in the statement.

	\vspace{0.5em}

	\noindent $\mathrm{(iv)}$
	The second part can be proved in the same way as the first part using Lemma \ref{Equiv:LandYcontinuous} instead of Lemma \ref{Equiv:LandYmeasurable}.
\endproof

\section{Conclusion}
	{We have introduced an exit contract design problem in this work}, where the principal provides a universal exit contract to multiple heterogeneous agents.
	Under a technical monotone condition, {we have developed
	a systematic technique to transform the exit contract design problem} into an optimal control problem,
	which can be formulated equivalently as a multiple optimal stopping problem, and it can be considered
	as a first best problem.
	This general method can be easily adapted to the discrete-time setting, and for some variations of the problem where the contract is required to be a Markovian and/or continuous functional of the underlying process.
	
	\vspace{0.5em}
	
	{An interesting future topic would be the mean-field extension of the exit contract design problem},
	where the agents interact between each other and the number of agents turns to infinite.
	To achieve this, one should rely on a mean-field extension of the Bank-El Karoui's representation theorem.
	{Another interesting topic would be the application of} our general approach to study more concrete economic problems,
	such as the optimal subsidy/tax policy problem for the government,
	where one looks for a universal subsidy/tax policy for different electricity companies in order to encourage them to replace the traditional fossil-fuel electricity generators by green energy ones.

\appendix

\section{Appendix}

\subsection{The Snell envelop approach to the optimal stopping problem}

	We recall here  some basic facts in the classical optimal stopping theory.
	We also refer to El Karoui \cite{ElKarouiSF}, Appendix D in Karatzas and Shreve \cite{KaratzasShreve}, Theorem 1.2 in Peskir and Shiryaev \cite{PeskirShiryaev} for a detailed presentation of the theory.
	
	\vspace{0.5em}
	
	Let $(\Om,\Fc,\P)$ be a completed probability space equipped with the filtration $\F = \{\Fc_t\}_{t\in [0,T]}$ satisfying the usual conditions.
	Let $G = (G_t)_{t \in [0,T]}$ be an optional process in class $(D)$, we first consider the following optimal stopping problem
	\begin{equation} \label{eq:optimal_stopping_recall}
		\sup_{\tau \in \Tc} \E \big[ G_{\tau} \big].
	\end{equation}
	Let $\pi = (t_j)_{j=0, \cdots, m}$ be a partition of $[0,T]$, i.e. $0 = t_0 < \cdots < t_m = T$.

	\begin{theorem} \label{thm:SnellEnv}
		$\mathrm{(i)}$
        There exists a (unique) strong supermartingale $S = (S_t)_{t \in [0, T]}$ (which is l\`adl\`ag) such that
		$$
			S_{\tau}
            ~=~
            \esssup_{\sigma \in \Tc_{\tau}}
            \E \big[ G_\sigma \big| \Fc_\tau \big],
			~\mbox{a.s.}
		$$
		Moreover, a stopping time $\tauh$ is an optimal solution to \eqref{eq:optimal_stopping_recall} if and only if
		$$
			(S_{\tauh \wedge t})_{t \in [0,T]} ~
            \mbox{is a martingale, and}~
			S_{\tauh} = G_{\tauh}, ~\mbox{a.s.}
		$$
		
		\noindent $\mathrm{(ii)}$
		Assume in addition that, either $G$ is USCE, or $G$ satisfies $G_t = G_{t_j}$ for $t \in [t_j, t_{j+1})$, $j=0, \cdots, m-1$.
		Then there exist optimal stopping times to \eqref{eq:optimal_stopping_recall}, and the smallest one is given by
		$$
			\tauh
			~:=~
			\inf \big\{ t \ge 0 ~: S_t = G_t \big\}
			~=~
			\essinf \big\{ \tau ~: S_{\tau} = G_{\tau} \big\}.
		$$
		Moreover, when $G$ satisfies $G_t = G_{t_j}$ for $t \in [t_j, t_{j+1})$, $j=0, \cdots, m-1$, one has $\tauh \in \{t_0, \cdots, t_m\}$, a.s.
		
	\end{theorem}

	Given a sequence  $\{G^i \}_{i=1}^n$ of optional process on $[0,T]$ in class $(D)$,
	we next consider the following multiple optimal stopping problem:
	\begin{equation}\label{eq:OptimalMultiStoppingRecall}
		Z_0
		~:=
		\sup \Big\{  \E \Big[\sum_{i = 1}^{n}G^i_{\tau_i} \Big] ~:  \{\tau_i\}_{i=1}^n \subset \Tc, ~ \tau_1 \leq \cdots \leq \tau_n \Big\}.
	\end{equation}

	\begin{theorem}\label{thm:OptimalMultiStoppingRecall}
		Assume that  the map $t \longmapsto \E \big[ G^i_t\big]$ is continuous for each $i=1, \cdots,n $.
		Then there exists an optimal solution $\{\tauh_i\}_{i=1}^n$ to the multiple stopping problem \eqref{eq:OptimalMultiStoppingRecall}.
	\end{theorem}

	\begin{proof}
        We will only consider the case $n = 2$, while the arguments for the general case are almost the same.

	\vspace{0.5em}

        Let $J$ be the optional process (see Theorem \ref{thm:SnellEnv} for its existence) such that
        $$
		J_{\tau_1} = \esssup_{\tau_2 \in \Tc_{\tau_1}}
		\E[G^2_{\tau_2}|\Fc_{\tau_1}]
		+ G^1_{\tau_1},
		~\mbox{a.s.}
	$$
	By Theorem \ref{thm:SnellEnv}, the above problem has an optimal solution $\tauh_2 = \tauh_2(\tau_1) \in \Tc_{\tau_1}$ for each $\tau_1 \in \Tc$.
	
	\vspace{0.5em}
	
	Further, by the dynamic programming principle, it is easy to check that
	$Z_0 = \sup_{\tau_1 \in \Tc}\E[J_{\tau_1}]$.
	At the same time, as $t \longmapsto \E[G^i_{\tau_i}] ~i = 1, 2$ are continuous,
	it follows by e.g. Proposition 1.6 and 1.5 in \cite{KobylanskiQuenezRouy-Mironescu2011} that the mapping $t \longmapsto \E[J_t]$ is also continuous.
	Therefore, there exists an optimal stopping time $\tauh_1 \in \Tc$ for the problem
	\begin{equation}\label{eq:n=2OptimalStoppingExistence}
		\E[J_{\tauh_1}] = \sup_{\tau_1 \in \Tc}\E[J_{\tau_1}],
	\end{equation}
	so that $(\tauh_1, \tauh_2(\tauh_1))$ is an optimal solution to \eqref{eq:OptimalMultiStoppingRecall}.
	\end{proof}

\subsection{Bank-El Karoui's representation of stochastic processes}

	We recall here the Bank-El Karoui's representation theorem for stochastic processes, with some slight modifications, together with some trivial extensions.
	Let us stay in the context with a completed probability space $(\Om,\Fc,\P)$,
	equipped with the filtration $\F = \{\Fc_t\}_{t\in [0,T]}$ satisfying the usual conditions.
	Let $\mu$ be a finite measure on $[0,T]$ and $h: [0,T] \x \Om \x \R \longrightarrow \R$.

	\begin{assumption}\label{Assumption:BErepresentation}
		For all $(t, \om) \in [0,T] \x \Om$, the map $\ell \longmapsto h(t,\om, \ell)$ is continuous and strictly increasing from $-\infty$ to $+\infty$.
		Moreover, for each $\ell \in \R$, the process $(t, \om) \longmapsto h(t, \om, \ell)$ is progressively measurable, and satisfies
		$$
			\E \bigg[\int_{[0,T]} \big| h(t, \ell) \big| \mu(dt) \bigg] < \infty.
		$$
	\end{assumption}

	\begin{theorem}[Bank-El Karoui's representation, \cite{BankKaroui2004, BankFollmer, Bank2018}] \label{thm:BErepresentationTheorem}
		Let Assumption \ref{Assumption:BErepresentation} hold true, $Y$ be an optional process of class (D) with $Y_T = \xi$ for some random variable $\xi$,
		and one of the following conditions holds true:
		\begin{itemize}
			\item either $\mu$ is atomless, and  $Y$ is USCE;
			
			\item or, for some $0 =t_0 < \cdots < t_m = T$ and $(c_j)_{j=1, \cdots, m}$ with $c_j > 0$, one has $\mu(dt) = \sum_{j=0}^m c_j \delta_{t_j}(dt)$,
			and $Y_t = Y_{t_j}$, $\Fc_t = \Fc_{t_j}$ for all $t \in [t_j, t_{j+1})$, $j=0, \cdots, m-1$.
		\end{itemize}
		Then,

		\vspace{0.5em}	
		
		\noindent $\mathrm{(i)}$ there exists an optional process $L : [0,T] \x \Om \longrightarrow \R$, such that
		\begin{equation*}
			\E\bigg[\int_{(\tau,T]} \Big|h\Big( t, \sup_{v\in [\tau,t)}L(v)\Big)\Big| \mu(dt) \bigg] < \infty,
			~~\mbox{for all}~ \tau \in \Tc,
		\end{equation*}
		and
		\begin{equation}\label{eq:BErepresentation}
			Y_\tau = \E\bigg[\int_{(\tau,T]} h\Big(t, \sup_{s \in [\tau,t)}L_s \Big) \mu(dt) + \xi \Big| \Fc_\tau\bigg],
			~\mbox{a.s., for all}~
			\tau \in \Tc.
		\end{equation}
		
		\noindent $\mathrm{(ii)}$  for any optional process $L$ gives the representation \eqref{eq:BErepresentation}, and any $\ell \in \R$,
		the stopping time
		\begin{equation} \label{eq:def_tau_l}
			\tau_{\ell} ~:=~  \inf \big\{ t\geq 0 ~: L_t \geq \ell \big\}
		\end{equation}
		is the smallest solution of the optimal stopping problem:
		\begin{equation}\label{OptimalStopping:BErepresentation}
			\sup_{\tau \in \Tc}
			~\E\bigg[Y_\tau + \int_{[0, \tau]}h(t, \ell) \mu(dt) \bigg].
		\end{equation}
	\end{theorem}

	\begin{remark}
		$\mathrm{(i)}$
		The above statement differs slightly with the original theorem in \cite{BankKaroui2004}.
		First, we let $Y_T = \xi$ in place of $Y_T = 0$ in  \cite{BankKaroui2004}.
		More importantly, we change the interval of integration from $[\cdot,\cdot)$ to $(\cdot,\cdot]$, the term $\sup_{v \in [\cdot,\cdot]}L_v$ to $\sup_{v \in [\cdot,\cdot)}L_v$.
		The main motivation is to unify the presentation when we embed the discrete-time process in the continuous-time setting as a right-continuous and left-limit piecewise constant process.
		The proof stays almost the same as that in \cite{BankKaroui2004}, and for completeness, we provide a sketch of proof.

		\vspace{0.5em}

		\noindent $\mathrm{(ii)}$
		In \cite{BankKaroui2004}, it is proved that there exists an optional process $L$ providing the representation \eqref{eq:BErepresentation}, and the associated stopping time $\tau_{\ell}$ is the smallest solution to the optimal stopping problem \eqref{OptimalStopping:BErepresentation}.
		Such a process $L$ may not be unique.
		By (trivially) extending Theorem 2 of Bank and F\"ollmer \cite{BankFollmer}, we provide the additional fact that for any optional process $L$ giving \eqref{eq:BErepresentation}, the associated stopping time $\tau_{\ell}$ is the smallest solution to \eqref{OptimalStopping:BErepresentation}.
		In particular, the stopping time $\tau_{\ell}$ does not depend on the choice of $L$.

		\vspace{0.5em}

		\noindent $\mathrm{(iii)}$
		Although the process $L$ giving the representation \eqref{eq:BErepresentation} is not unique,
		it has a unique maximum solution in the following sense (see Theorem 2.16 of \cite{Bank2018}).
		For all $\tau \in \Tc$ and $\sigma \in \Tc_{\tau}$, let us denote by $l_{\tau,\sigma}$ the unique $\Fc_\tau$-measurable random variable satisfying
		$$
			\E \big[ Y_\tau - Y_\sigma \big| \Fc_\tau \big]
			~=~
			\E \Big[\int_{(\tau,\sigma]}h(t, \ell_{\tau,\sigma})\mu(dt) \Big| \Fc_\tau \Big],
			~\mbox{on}~\{\P[\mu([\sigma,\tau)|\Fc_\tau]>0\},
		$$
		and $\ell_{\tau,\sigma} = \infty$ on $\{\P[\mu([\sigma,\tau)|\Fc_\tau]=0\}$,
		and then define
		$$
			\widetilde L(\tau)
			~:=~
			\essinf_{\sigma \in \Tc_\tau}\ell_{\tau,\sigma}.
		$$
		Then $(\widetilde L( \tau))_{\tau \in \Tc}$ can be aggregated into an optional process $\widetilde L$,
		and it is the maximum solution to \eqref{eq:BErepresentation} in the sense that for any optional process $L$ satisfying the representation \eqref{eq:BErepresentation},
		one has $L_\tau \leq \widetilde L_\tau$ for all $\tau \in \Tc$.
	\end{remark}

	\proof{\bf of Theorem \ref{thm:BErepresentationTheorem}.}
	We will prove the theorem separately for the continuous setting (i.e. $\mu$ is atomless and $Y$ is USCE) and the discrete-time setting (i.e. $\mu$ is sum of Dirac measures).
	
	\vspace{0.5em}

	\noindent $\mathrm{(i.a)}$
	In the continuous-time setting where $\mu$ is atomless and $Y$ is USCE,
	one can assume $Y_T = 0$ by considering  the process $(Y_t - \E[Y_T|\Fc_t])_{t \in [0,T]}$ in place of $Y$.
	Then by \cite{Bank2018}, there exists an optional process $L : [0,T] \x \Om \longrightarrow \R$, such that, for all $\tau \in \Tc$,
	\begin{equation*}
		\E\bigg[\int_{[\tau,T)} \Big|h\Big( t, \sup_{v\in [\tau,t]}L(v)\Big)\Big| \mu(dt) \bigg] < \infty,
		~\mbox{and}~
		Y_\tau = \E\bigg[\int_{[\tau,T)} h\Big(t, \sup_{s \in [\tau,t]}L_s \Big) \mu(dt) \Big| \Fc_\tau\bigg].
	\end{equation*}

	As $\mu$ is atomless, one can change the interval of integration from $[\tau, T)$ to $(\tau, T]$.
	Moreover, for the term $\sup_{s \in [\tau,t]}L_s$, the fact that $t \mapsto \sup_{s \in [\tau,t]} L_s$ is increasing implies that
	$\sup_{s \in [\tau,t]}L_s = \sup_{s \in [\tau,t)}L_s$, $\mu(dt)$-a.s.
	We hence obtain \eqref{eq:BErepresentation}.

	\vspace{0.5em}

	\noindent $\mathrm{(i.b)}$
	In the discrete-time setting, the representation \eqref{eq:BErepresentation} can be proved easily by a backward induction argument.
	However, to unify the proof of Item $\mathrm{(ii)}$ in both continuous-time and discrete-time setting,
	we provide the proof of \eqref{eq:BErepresentation} as in the continuous-time setting.
	 When $\mu$ is sum of Dirac measures, we denote $\pi := (t_j)_{0 \leq j \leq m}$ of $[0,T]$ and write $\mu^\pi(dt) := \sum_{j = 1}^mc_j\delta_{t_j}(dt)$ instead of $\mu$ to remind the difference of the context.
	Let $Y$ be an optional process such that $Y_T = \xi$.
	We define the optional process $Z^\ell$ such that
	\begin{equation}\label{def:SnellEnvelop}
		Z^\ell_{\tau}
		~ := ~
		\esssup_{\sigma \in \Tc_{\tau}}
		\E\bigg[Y_\sigma  + \int_{(\tau,\sigma]}
		h(t, \ell)
		\mu^\pi(dt) \bigg| \Fc_{\tau}\bigg],
		~\forall~\tau \in \Tc,
	\end{equation}
	and further a family of stopping times $\tau^\ell$ as well as a process $L$:
	\begin{align*}
		\tau^\ell_{t}
		~ := ~
		\min\{s \geq t: Z^\ell_{s} = Y_{s}\},
		~
		L_t
		~ := ~
		\sup\{l \in \R:Z^\ell_{t} = Y_t\}.
	\end{align*}
	We will then prove that $L$ is the required optional process by a backward induction argument.

 	\vspace{0.5em}

	For a fixed $j \in \{0, 1, \cdots, m-1\}$, we assume that $L$ satisfies
	\begin{equation*}
		Y_{t_k} = \E\bigg[\xi + \int_{(t_k,T]}h(t,\sup_{v \in [t_k,t)}L_v)\mu^{\pi}(dt) ~\bigg|~ \Fc_{t_k}\bigg],
		~\mbox{for}~k = j + 1, \cdots, m.
	\end{equation*}
	We will prove that $Y_{t_j}$ can also be represented by $L$ as above.

 	\vspace{0.5em}

	Let us observe that
	$
		\{\tau^\ell_{t_j} < t_i\}
		~=~
		\big\{ \ell \leq \sup_{v \in [t_j, t_i)}L_v \big \},
	$
	which gives the following properties:
    \begin{align*}
		&
        \{\tau^\ell_{t_j} = {t_j}\}
		~ = ~
        \{\tau^\ell_{t_j} < t_{j + 1}\}
        ~ = ~
		\{ \ell \leq L_{t_j} \}
		,~
        \tau^{L_{t_j}+}_{t_j} = \tau^{L_{t_j}}_{t_j} = t_j,
        \\ &
        \{\tau^\ell_{t_j} = t_k\}
        ~ = ~
        \{t_k \leq \tau^\ell_{t_j} < {t_{k + 1}}\}
        ~ = ~
        \{\sup_{v\in [t_j,t_k)}L_v
            <
        \ell
            \leq
          \sup_{v \in [t_j,t_{k + 1})}L_v \}.
    \end{align*}
    and that on $
            H_{j,k} : =
            \{\sup_{v \in [t_j,t_k) }L_v
                        \leq
                L_{t_j}
                        <
            \sup_{v \in [t_j,t_{k + 1}) }L_v  \}$,
    \begin{align*}
        L_{t_j}
        ~ = ~
        \sup_{v \in [t_j,t) }L_v~\mbox{for}~t < t_{k+1},
        ~
        \sup_{v \in [t_k,t) }L_v
        ~ = ~
        \sup_{v \in [t_j,t) }L_v~\mbox{for}~t \geq t_{k + 1}.
    \end{align*}
    Now we claim that $\ell \longmapsto Z^\ell_i $ is continuous a.s. for each $i$.
    In fact, for any $\ell_1$, $\ell_2$ with $\ell_1 \leq \ell_2$,
    \begin{align*}
        Z^{\ell_2}_{t_j}
        ~ \geq ~
        Z^{\ell_1}_{t_j}
        ~ & \geq ~
        \E\bigg[Y_{\tau^{\ell_2}_{t_j}}  + \int_{(t_j,\tau^{\ell_2}_{t_j}]}
            h(t,\ell_1)
        \mu^\pi(dt)
        \Big| \Fc_{t_j}\bigg]
        \\
        ~ & \geq ~
        Z^{\ell_2}_{t_j}
        -
        \E\bigg[\int_{({t_j},T]}
            \Big|
                h(t,\ell_1)
                -
                h(t,\ell_2)
            \Big|\mu^\pi(dt)
        \bigg| \Fc_{t_j}\bigg].
    \end{align*}
    Then by the dominated convergence theorem, we have
    $$
        \lim_{\ell \to \ell_1}Z^{\ell}_{t_j} = Z^{\ell_1}_{t_j},
        ~\mbox{a.s.}
    $$

	On the other hand, it follows by the definition of $\tau^{\ell}_t$ and and the representation for $Y_{t_k}$, $k = j+1, \cdots, m$,
	together with the tower property that
    \begin{align*}
        Z^\ell_{t_j}
        ~ = & ~
        \E\bigg[
            Y_{t_j} \mathds{1}_{\{\tau^\ell_{t_j} = t_j\}}
                +
            \bigg(Y_{\tau^\ell_{t_j}} + \int_{(t_{j},\tau^\ell_{t_j}]}
                h(t,\ell)
            \mu^\pi(dt)\bigg)
            \mathds{1}_{\{\tau^\ell_{t_j} > t_j\}}
        \Big| \Fc_{t_j}\bigg]
        \\
        ~ = & ~
        \E\bigg[
            Y_{t_j}  \mathds{1}_{\{\tau^\ell_{t_j} = t_j\}}
                 +
            \bigg(
                \xi + \int_{(\tau^\ell_{t_j},T]}
                    h(t,\sup_{v \in [\tau^\ell_{t_j}, t)}L_v)
                    \mu^\pi(dt)
                +
                \int_{(t_{j},\tau^\ell_{t_j}]}
                    h(t,\ell)\mu^\pi(dt)
            \bigg)
            \mathds{1}_{\{\tau^\ell_{t_j} > t_j\}}
        \Big| \Fc_{t_j}\bigg]
        \\
        ~ = & ~
        \E\bigg[
            Y_{t_j} \mathds{1}_{\{L_{t_j} \geq \ell\}}
                  +
           \sum_{k = j + 1}^{m}
            \bigg(
                \xi + \int_{(t_k,T]}
                    h(t,\sup_{v \in [t_k, t)}L_v)
                    \mu^\pi(dt)
                \\ & \qquad \qquad \qquad \qquad \qquad \quad +
                \int_{(t_{j},t_k]}
                    h(t,\ell)\mu^\pi(dt)
            \bigg)
                \mathds{1}_{\{\tau^\ell_{t_j} = t_k\}}
        \Big| \Fc_{t_j}\bigg].
    \end{align*}
    Since $Z^{L_{t_j}}_{t_j} = Z^{L_{t_j}+}_{t_j}$ holds true by the continuity of $\ell \longmapsto Z^\ell_{t_j} $, we have that
    \begin{align*}
        Y_{t_j}
        ~ &  =  ~
        \lim_{\ell \to L_{t_j}+}
            \E\bigg[
                      \sum_{k = j + 1}^{m}
            \bigg(
                \xi + \int_{(t_k,T]}
                    h(t,\sup_{v \in [t_k, t)}L_v)
                    \mu^\pi(dt)
                +
                \int_{(t_{j},t_k]}
                    h(t,\ell)\mu^\pi(dt)
            \bigg)
                \mathds{1}_{\{\tau^\ell_{t_j} = t_k\}}
        \Big| \Fc_{t_j}\bigg]
        \\
        ~ & =   ~
        \E\bigg[
           \sum_{k = j + 1}^{m}
            \bigg(
                \xi + \int_{(t_k,T]}
                    h(t,\sup_{v \in [t_k, t)}L_v)
                    \mu^\pi(dt)
                +
                \int_{(t_{j},t_k]}
                    h(t,L_{t_j})\mu^\pi(dt)
            \bigg)
                \mathds{1}_{H_{j,k}}
        \Big| \Fc_{t_j}\bigg]
        \\
        ~  & = ~
        \E\bigg[
           \sum_{k = j + 1}^{m}
            \bigg(
                \xi + \int_{(t_k,T]}
                    h(t,\sup_{v \in [t_j, t)}L_v)
                    \mu^\pi(dt)
                +
                \int_{(t_{j},t_k]}
                    h(t,\sup_{v \in [t_j, t)}L_v)\mu^\pi(dt)
            \bigg)
                \mathds{1}_{H_{j,k}}
        \Big| \Fc_{t_j}\bigg]
        \\
        ~ & = ~
        \E\bigg[
           \sum_{k = j + 1}^{m}
            \bigg(
                \xi +
                \int_{(t_j,T]}
                    h(t,\sup_{v \in [t_j,t)}L_v)
                \mu^{\pi}(dt)
            \bigg)
                \mathds{1}_{H_{j,k}}
        \Big| \Fc_{t_j}\bigg]
        \\
        ~ & = ~
        \E\bigg[
            \xi +
            \int_{(t_j,T]}
            h(t,\sup_{v \in [t_j,t)}L_v)
            \mu^{\pi}(dt)
            \Big| \Fc_{t_j}\bigg].
    \end{align*}

\noindent $\mathrm{(ii)}$
    The optimal stopping theory implies that
    it is sufficient to prove
        $Z^{\ell}_{\tau_\ell} = Y_{\tau_\ell}$ a.s.
    and for any $0 \leq t < \tau_\ell$, $Z^{A,\ell}_{t} > Y_t$ a.s.,
    where $Z^{\ell}$ is defined as \eqref{def:SnellEnvelop}.
    We observe that, for each $\ell \in \R$, one has
	$$
		\{\tau_\ell < t\}
		~\subset~
		\Big\{\sup_{v\in [0,t)}L_v \geq \ell\Big\}
        ~=~
		\Big\{\sup_{v\in [\tau_i,t)}L_v \geq \ell\Big\}.
	$$
      While $Z^{\ell}_{\tau_\ell} \geq Y_{\tau_\ell}$ a.s. is trivial, to prove the reverse inequality, we have that
      \begin{align*}
        Z^{\ell}_{\tau_\ell}
        ~ & = ~
        \esssup_{\sigma \in \Tc_{\tau_\ell}}
            \E\bigg[Y_\sigma + \int_{(\tau_\ell,\sigma]}
                h(s,\ell)
            \mu(ds)\Big|\Fc_{\tau_\ell}\bigg] \\
        ~ & = ~
        \esssup_{\sigma \in \Tc_{\tau_\ell}}
            \E\bigg[ \xi + \int_{(\sigma,T]}
                h\Big(s,\sup_{v\in [\sigma,s)}L_v\Big)
            \mu(ds)
            +
            \int_{(\tau_\ell,\sigma]}
                h(s,\ell)
            \mu(ds)\Big|\Fc_{\tau_\ell}\bigg] \\
        ~ & \leq ~
        \esssup_{\sigma \in \Tc_{\tau_\ell}}
            \E\bigg[ \xi + \int_{(\sigma,T]}
                h\Big(s,\sup_{v\in [\tau_\ell,s)}L_v\Big)
            \mu(ds)
            +
            \int_{(\tau_\ell,\sigma]}
                h\Big(s,\sup_{v\in [\tau_\ell,s)}L_v\Big)
            \mu(ds)\Big|\Fc_{\tau_\ell}\bigg] \\
        ~ & = ~
            Y_{\tau_\ell}
        ~\mbox{a.s.}
      \end{align*}
      Then for any $0 \leq t < \tau_\ell$, we observe that
	  $$
      \Big\{\sup_{v\in[t,s)}L_v < \ell\Big\} \supset \{\tau_\ell > s\},
      $$
      which implies that on $\{t < \tau_\ell\} $,
      \begin{align*}
        Z^{\ell}_t
        ~ & = ~
        \esssup_{\sigma \in \Tc_t}
            \E\bigg[Y_\sigma + \int_{(t,\sigma]}
                 h(s,\ell)
           \mu(ds)\Big|\Fc_{t}\bigg]  \\
        ~ & \geq ~
            \E\bigg[ \xi + \int_{(\tau_\ell,T]}
                h\Big(s,\sup_{v\in [\sigma,s)}L_v\Big)
            \mu(ds)
            +
            \int_{(t,\tau_\ell]}
                h(s,\ell)
            \mu(ds)\Big|\Fc_{\tau_\ell}\bigg]  \\
        ~ & > ~
            \E\bigg[ \xi +  \int_{(t,T]}
                h\Big(s,\sup_{v\in [t,s)}L_v\Big)
           \mu(ds)\Big|\Fc_{t}\bigg] \\
        ~ & = ~
            Y_t,
        ~\mbox{a.s.}
      \end{align*}
      We hence conclude the proof.
	\endproof

\end{document}